\documentclass[12pt, leqno]{amsart}

\usepackage{amssymb, amsmath, verbatim}
\usepackage{graphicx, lscape}
\usepackage{booktabs}
\usepackage{multirow}

\numberwithin{figure}{section}
\newtheorem{theorem}{Theorem}[section]
\newtheorem{lemma}{Lemma}[section]
\newtheorem{corollary}{Corollary}[section]
\newtheorem{example}{Example}[section]
\newtheorem{assumption}{Assumption}[section]
\newcommand{\vare}{\varepsilon}

\newcommand{\lf}{\lfloor}
\newcommand{\rf}{\rfloor}

\newcommand\bae{\bar{e}}

\newcommand\bbe{\mbox{\boldmath${ \beta}$}}

\newcommand\btheta{\mbox{\boldmath${\theta}$}}
\newcommand\bTheta{{\mbox{\boldmath${ \Theta}$}}}

\newcommand\cS{\mbox{$\mathcal{S}$}}

\newcommand\bX{{\bf X}}
\newcommand\bG{{\bf G}}
\newcommand\T{\top}
\newcommand\bZ{{\bf Z}}
\newcommand\bA{{\bf A}}

\newcommand\bx{{\bf x}}

\def\beq{\begin{equation}}
\def\eeq{\end{equation}}
\def\bals{\begin{align*}}
\def\eals{\end{align*}}

\def\bal{\begin{align}}
\def\eal{\end{align}}

\numberwithin{equation}{section}
\numberwithin{theorem}{section}
\numberwithin{table}{section}

\DeclareMathOperator{\ar}{AR}
\DeclareMathOperator{\garch}{GARCH}

\def\convP{\stackrel{\mbox{$\scriptstyle P$}}{\rightarrow}}

\begin{document} 

\title[Change point detection in heteroscedastic time series]
{Change point detection in heteroscedastic time series}
\author {Tomasz G\'orecki}
\address{Tomasz G\'orecki, Faculty of Mathematics and Computer Science,
Adam Mickiewicz University, 61-614 Pozna{\'n}, Poland}
\email{tomasz.gorecki@amu.edu.pl}
\author {Lajos Horv\'ath}
\address{Lajos Horv\'ath, Department of Mathematics, University of Utah,
Salt Lake City, UT 84112--0090, USA}
\email{horvath@math.utah.edu}

\author{Piotr Kokoszka}
\address{Piotr Kokoszka, Department of Statistics, Colorado State University,
Fort Collins, CO 80523--1877, USA}
\email{Piotr.Kokoszka@colostate.edu}

\thanks{This project was  supported by NSF grants DMS 1305858
at the University of Utah and DMS 1462067 at Colorado State University.}

\begin{abstract} Many  time series exhibit changes both in
 level and in variability. Generally, it is more important to detect a change
in the level, and changing or smoothly evolving variability can confound
existing tests. This paper develops a framework for testing for shifts
in the level of a series which accommodates the possibility of changing
variability. The resulting tests are robust both to heteroskedasticity and
serial dependence. They rely on a new functional central limit
theorem for  dependent random variables whose variance can change
or trend in a  substantial way. This new result is of independent interest
as it can be applied in many  inferential contexts applicable to
time series. Its application  to change point tests relies on
a new approach which utilizes Karhunen--Lo{\'e}ve expansions of
the limit Gaussian processes. After presenting the theory in
the most commonly encountered setting of the detection of
a change point in the mean, we show how it can be extended
to linear and nonlinear regression. Finite sample performance is
examined by means of a simulation study and an application to
yields on US treasury bonds.



\end{abstract}

\maketitle

\section{Introduction }\label{sec-main}
In the most common change point paradigm, we consider the model
\[
X_i=\mu_i+u_i,\quad 1\leq i \leq N,
\]
with mean zero errors, $Eu_i=0$, and
wish to test the no change in the mean  null hypothesis
\[
H_0:\;\;\;\mu_1=\mu_2=\ldots =\mu_N.
\]
The general alternative is that $H_0$ does not hold, but we
target several  change point alternatives discussed in Examples
\ref{ex-1}--\ref{ex-3}.  Cs\"org\H{o} and
Horv\'ath (1997) provide an account of early results in change point
detection based mainly on independent and identically distributed
error terms and connect the likelihood method to maximally selected
CUSUM. Aue and Horv\'ath (2013) explain the extension of some of the
classical results to time series setting. Jeng (2015) provides an
overview of change point detection in finance. In change point
research, usually the homoscedasticity of the errors is assumed.
Incl\'an and Tiao (1994), Gombay and Horv\'ath (1994), Davis et al.\
(1995), Lee and Park (2001), Deng and Perron (2008), Antoch et al.\
(1997), Berkes et al.\ (2009), Aue et al.\ (2009), Wied et al.\ (2012,
2013) propose tests when the mean and/or the variance are changing
under the alternative, i.e.\ heteroscedastic errors can occur only
under the alternative.  Dalla et al.\ (2015) and Xu (2015)
point out that in some
applications the errors are heteroscedastic, which should be taken
into account when we test the validity of $H_0$.  Our paper is
related to  their work. It  rigorously derives a new class of
tests which are valid under weak assumptions, which do not require
any mixing conditions. Busetti and Taylor
(2004), Cavaliere et al. (2011), Cavaliere and Taylor (2008), Hansen
(1992) and Harvey et al. (2006) investigate change
point tests when some type of nonstationarity is exhibited by the data.

 The  paper is organized as follows.
 Section~\ref{s:alr} develops the  asymptotic framework. The limit
 distributions of the tests statistics are nonstandard if error
 heteroskedastisity is allowed. These distribution can however be
 computed using suitable Karhunen--Lo{\'e}ve expansions, which also
 lead to practical ways of computing the critical or P--values, as
 explained in Section~\ref{s:comp}.  Section~\ref{sec-sim} explores
 the finite sample performance of the tests. Proofs of the asymptotic
 results are collected in Section~\ref{sec-basicpr}.

\section{Assumptions and limit results} \label{s:alr}
We consider heteroscedastic errors:
 \begin{assumption}\label{as-1}
$u_i=u_{i,N}=a(i/N)e_i,\;\; 1\leq i \leq N$,
 \end{assumption}
where the function $a$ satisfies
\begin{assumption}\label{as-2} $a(t), 0\leq t \leq 1$,
 has bounded variation on $[0,1]$.
 \end{assumption}
 (For the definition and properties of functions with bounded variation we
 refer e.g.\ to Hewitt and Stromberg (1969).)
 \medskip

We allow a very general class of errors $e_i, -\infty<i<\infty$:

 \begin{assumption}\label{as-3} $Ee_i=0$ and $e_i, 1\leq i <\infty,$ satisfy the functional central limit theorem, i.e.\ there is $\sigma>0$ such that
 $$
 N^{-1/2}\sum_{\ell=1}^{\lf Nt\rf}e_\ell \;\;\stackrel{{\mathcal D}[0,1]}{\longrightarrow}\;\;\sigma W(t),
 $$
 where $W(t), 0\leq t <\infty$, denotes a Wiener process (standard Brownian motion).
 \end{assumption}
We do not assume stationarity or any
form of mixing for the error terms, they must merely satisfy
a Central Limit Theorem, which is a minimal requirement for
the existence of an asymptotic distribution of common test statistics.

The theory of testing in the various contexts studied below is
based on the following result.

 \begin{theorem}\label{basic} If Assumptions \ref{as-1}--\ref{as-3} are
satisfied, then
\[
 N^{-1/2}\sum_{\ell=1}^{\lf Nt\rf}u_\ell\;\stackrel{{\mathcal D}[0,1]}{\longrightarrow}\;W(b(t)),
\]
 where $W(u), 0\leq u <\infty$, is a Wiener process
 (standard Brownian motion) and
\[
 b(t)=\sigma^2\int_0^t a^2(u)du.
\]
 \end{theorem}

 Theorem \ref{basic} is a major theoretical contribution of this
 paper.  It establishes the asymptotic behavior of the partial sum
 process for dependent random variables with evolving variance without
 imposing any stationarity or mixing conditions on the errors $e_i$.
 The time transformed Wiener process has been known to appear as a
 limit since the 1950's. Limit theorems similar to our
 Theorem~\ref{basic} under mixing assumptions are discussed in Hall and
 Heyde (1980) and Davidson (1994). We show that mixing conditions are
 actually not needed.  Theorem \ref{basic} can be used in settings
 that extend beyond change--point detection, for example in various
 unit root and trend tests.

In Section~\ref{ss:cm},  we show how Theorem~\ref{basic} leads to
a class of change point tests in the setting of a potential
change in mean. Section~\ref{ss:cr} extends the scope of applicability
to regression models.

\subsection{Change point in the mean} \label{ss:cm} ~\\
We begin by presenting several examples of changes our tests can
detect.

\begin{example}\label{ex-1} We say that there is exactly one change in the mean if $\mu_i=\tilde{\mu}_1, 1\leq i\leq \lf N\theta\rf$, and $\mu_{i}=\tilde{\mu}_2, \lf N\theta\rf<i\leq N$, where $\tilde{\mu}_1\neq \tilde{\mu}_2$ and $0<\theta<1$. In this case $\lf N\theta\rf$ is the time of change and
 \begin{displaymath}
 d(t)=\left\{
 \begin{array}{ll}
 t\tilde{\mu}_1-t(\theta\tilde{\mu}_1+(1-\theta)\tilde{\mu}_2),\;\;&\mbox{if}\;\;0\leq t \leq \theta,
 \vspace{.3cm}\\
 \theta\tilde{\mu}_1+(t-\theta)\tilde{\mu}_2-t(\theta\tilde{\mu}_1+(1-\theta)\tilde{\mu}_2),\;\;&\mbox{if}\;\;\theta\leq t \leq 1.
 \end{array}
 \right.
 \end{displaymath}
 \end{example}

 \begin{example}\label{ex-2} If $\mu_i=\tilde{\mu}_\ell, \lf N\theta_{\ell-1}\rf<i\leq \lf N\theta_{\ell}\rf$, $1\leq \ell\leq m+1$, $\theta_0=0, \theta_{m+1}=N$ and
 $\mu_\ell\neq \mu_{\ell'}$ for some $\ell\neq \ell'$, we have at most $m$ changes in the mean. With $\tilde{\mu}=(\theta_1-\theta_0)\tilde{\mu}_1+(\theta_2-\theta_1)\tilde{\mu}_2+\ldots +(\theta_{m+1}-\theta_m)\tilde{\mu}_{m+1}$, we get

 \begin{displaymath}
 d(t)=\left\{
 \begin{array}{ll}
 t\tilde{\mu}_1-t\tilde{\mu},\;\;\mbox{if}\;\;0\leq t \leq \theta_1,
 \vspace{.3cm}\\
 \theta_1\tilde{\mu}_1+(t-\theta_1)\tilde{\mu}_2-t\tilde{\mu},\;\;\mbox{if}\;\;\theta_1\leq t \leq \theta_2,
 \\
 \hspace{1cm}\vdots
 \\
 \theta_1\tilde{\mu}_1+(\theta_2-\theta_1)\tilde{\mu}_2+\ldots +(t-\theta_m)\tilde{\mu}_{m+1}-t\tilde{\mu},\;\;\mbox{if}\;\;\theta_m\leq t \leq \theta_{m+1}.
 \end{array}
 \right.
 \end{displaymath}
 \end{example}

 \begin{example}\label{ex-3}  Let $\tilde{d}(t)$ be a continuous function   on $[0,1]$ and define $\mu_i=\tilde{\mu}_1$, if $1\leq i\leq \lf N\theta\rf$ and $\mu_i=\tilde{d}(i/N)$, if $\lf N\theta\rf <i\leq N$. If $\tilde{d}(t)$ is different from $\tilde{\mu}_1$, the mean is constant before $\lf N\theta \rf$ but it is determined by $\tilde{d}(t)$
 after the time of change. Now
 \begin{displaymath}
 d(t)=\left\{
 \begin{array}{ll}
\displaystyle  t\tilde{\mu}_1-t\left(\theta\tilde{\mu}_1+\int_\theta^1\tilde{d}(u)du\right),\;\;&\mbox{if}\;\;0\leq t \leq \theta,
 \vspace{.3cm}\\
\displaystyle  \theta\tilde{\mu}_1+\int_\theta^t\tilde{d}(u)du-t\left(\theta\tilde{\mu}_1+\int_\theta^1\tilde{d}(u)du\right),\;\;&\mbox{if}\;\;\theta\leq t \leq 1.
 \end{array}
 \right.
 \end{displaymath}
 This example includes linearly or polynomially increasing/decreasing means after the change.
 \end{example}

Recall the definition of the CUSUM process:
\[
Z_N(t)=N^{-1/2}\left(\sum_{\ell=1}^{\lf Nt\rf}X_\ell-\frac{\lf Nt\rf}{N}\sum_{\ell=1}^N X_\ell\right).
\]
In the setting of iid normal errors, the maximally
selected CUSUM statistic can be derived from the maximum likelihood
principle. Test  based on other functionals of the CUSUM process
are often the simplest and most effective in more general settings,
and are, in fact,  the most often used change point detection procedures.
Our testing procedures are based on functionals of the CUSUM process
as well. However, in the setting specified by Assumptions
\ref{as-1}--\ref{as-3}, especially~\ref{as-1}, the asymptotic behavior
of this process is very different than in the usual case of homoskedastic
errors. Understanding this behavior is necessary to derive the tests.
The weak convergence of the CUSUM process is an immediate consequence
of Theorem \ref{basic}.

 \begin{corollary}\label{basic-cor} If $H_0$ and Assumptions \ref{as-1}--\ref{as-3} are satisfied, then we have that
\[
 Z_N(t)\;\stackrel{{\mathcal D}[0,1]}{\longrightarrow}\;\Gamma(t),\;\;\mbox{where}\;\;\Gamma(t)=W(b(t))-tW(b(1)).
\]
 \end{corollary}

 \medskip It is easy to see that $\Gamma(t), 0\leq t \leq 1,$ is a
 Gaussian process with $E\Gamma(t)=0$ and
\begin{equation} \label{e:Cts}
 C(t,s)=E\Gamma(t)\Gamma(s)=b(t\wedge s)-tb(s)-sb(t)+tsb(1).
\end{equation}
 where $t\wedge s=\min(t,s)$.

 In general, the computation of the distribution functions of
 functionals of the limit in Corollary \ref{basic-cor} is nearly
 impossible due to the dependence on the unknown function $b(t)$.
 However, combining the Karhunen--Lo\'eve expansion with principal
 component analysis we can approximate the distributions of $L^2$
 functionals. This is done in Section~\ref{s:comp}.

 The derivation of the $L^2$ functional of the {\em standardized }
 $Z_N(t)$ requires a bound on the correlation between the elements of
 the sequence $\{e_i, -\infty<i<\infty\}$:

 \begin{assumption}\label{as-cor}\;\;
 $ |Ee_ie_k|=O((|k-i|+1)^{-\kappa}),\;\;\mbox{with some}\;\; \kappa>1.$
 \end{assumption}

 \begin{corollary}\label{cor-sq} If $H_0$ holds and Assumptions \ref{as-1}--\ref{as-3} are satisfied, then we have that
 \begin{equation}\label{cv}
 \int_0^1 Z_N^2(t)dt\;\;\stackrel{{\mathcal D}}{\to}\;\;
\int_0^1\Gamma^2(t)dt.
 \eeq
 If, in addition, Assumption \ref{as-cor} also holds, then we have that
 \begin{equation}\label{ad}
 \int_{1/N}^{1-1/N} \frac{Z_N^2(t)}{t(1-t)}dt\;\;\stackrel{{\mathcal D}}{\to}\;\;\int_0^1\frac{\Gamma^2(t)}{t(1-t)}dt.
 \eeq
 \end{corollary}

 The statistic in \eqref{cv} is a version of the Cram\'er--von Mises
 statistic while \eqref{ad} is a modification of the Anderson--Darling
 statistic (cf.\ Shorack and Wellner (1986, p.\ 148)).

 We would like to note that $Z_N(t)$ is not ``symmetric" since by
 definition $Z_N(t)=0$, if $0<t<1/N$ while $|Z_N(t)|>0$ if
 $1-1/N<t<1$. However,
\[
 \tilde{Z}_N(t)=N^{-1/2}\left(\sum_{\ell=1}^{\lf (N+1)u\rf}X_\ell-\frac{\lf (N+1)u \rf}{N}\sum_{\ell=1}^NX_\ell\right)
\]
is ``tied down" in a neighborhood of 0 as well as 1. Relation
\eqref{ad} can be replaced with
 \begin{equation}\label{ad-m}
 \int_{0}^{1} \frac{\tilde{Z}_N^2(t)}{t(1-t)}dt\;\;\stackrel{{\mathcal D}}{\to}\;\;\int_0^1\frac{\Gamma^2(t)}{t(1-t)}dt.
 \end{equation}

 We conclude this section by establishing the asymptotic behavior of
 the tests statistics when $H_0$ does not hold. Let
\[
 d_k=\sum_{\ell=1}^k(\mu_\ell-\bar{\mu}), \;\;1\leq k \leq N,\;\;\mbox{with}\;\;\bar{\mu}=\frac{1}{N}\sum_{\ell=1}^N\mu_\ell,
\]
be the drift term of the CUSUM process.  We impose the following minor
restrictions on the expected values of the observations:

 \begin{assumption}\label{mu-con} \quad $\max_{1\leq k \leq N}|\mu_k|=O(1)$ and
 $$
 d(u)=\lim_{N\to \infty}\frac{d_{\lf Nu\rf}}{N},\quad 0\leq u \leq 1,
 $$
 exists. (We use $d_{\lf Nu\rf}=0$, if $0\leq u <1/N.$)
 \end{assumption}

 We would like to point out that Assumption \ref{mu-con} is
 automatically satisfied under $H_0$, since $d_k=0$ for all $1\leq k
 \leq N$, i.e.\ there is no drift.

The following theorem will be used to establish the consistency of the
tests.

 \begin{theorem}\label{alt-1} If   Assumptions \ref{as-1}--\ref{as-3}  and \ref{mu-con} are satisfied, then we have that
 \begin{equation}\label{alt-st-1}
 \frac{1}{N}\int_0^1 Z_N^2(t)dt\;\stackrel{P}{\to}\int_0^1 d^2(u)du.
 \eeq
 If, in addition, Assumption \ref{as-cor} also holds, then we have that
 \begin{equation}\label{alt-st-2}
 \frac{1}{N}\int_{1/N}^{1-1/N} \frac{Z_N^2(t)}{t(1-t)}dt\;\stackrel{P}{\to}\int_0^1 \frac{d^2(t)}{t(1-t)}dt.
 \eeq
 \end{theorem}

If
\begin{equation}\label{alt-as}
 \int_0^1 d^2(u)du>0
 \eeq
holds, then
$$
\int_0^1 Z_N^2(t)dt\;\;\;\stackrel{P}{\to}\;\infty\;\;\;\mbox{and}
\;\;\;\int_{1/N}^{1-1/N} \frac{Z_N^2(t)}{t(1-t)}dt\;\stackrel{P}{\to}\;\;
\infty.
$$
However, to establish the consistency, we have to carefully examine
the asymptotic behavior of the estimated eigenvalues; this is
done in Sections \ref{sec-est-1} and  \ref{sec-est-2}. In the following,
condition \eqref{alt-as} is adopted as the formal definition of $H_A$.

\subsection{Change point in the parameters of linear and nonlinear
regression} \label{ss:cr} ~\\
Section~\ref{ss:cm}  focused on the theory of testing for
a change point in mean in order to explain the essence of our
approach and provide the details in that most common setting.
In this section,  we consider more general regression settings.
The proofs use suitably defined  model residuals which
approximate the unobservable heteroskedastic errors $u_i$.
The tests of Section~\ref{ss:cm}  correspond to the residuals
$\hat u_i = X_i - \bar X_N, \ 1 \le i \le N$. Once the
residuals are defined, the asymptotic
arguments parallel those used to establish the results of
Section~\ref{ss:cm}, so we just outline the proofs.

Consider first the usual  linear regression
\[
X_i = \bx_i^\top \bbe_i + u_i, \ \ \ 1 \le i \le N.
\]
We wish to test $H_0:\ \bbe_1 = \bbe_2 = \ldots \bbe_N$ against
the change point alternative. The following standard assumption is
made.
\begin{assumption} \label{a:xu-indep}
The sequences $\{ \bx_i \}$ and $\{ u_i \}$ are independent.
The sequence  $\{ \bx_i \}$ is stationary, ergodic,
 and $E\| \bx_0 \|^2 < \infty$.
\end{assumption}

We use the  least squares estimator $\hat\bbe_N=\bA_N^{-1}\bX_N$,
where
\[
\bA_N = \begin{bmatrix}
    \bx_{1}^\T      \\
    \vdots \\
    \bx_N^\T
\end{bmatrix}^\T
\begin{bmatrix}
    \bx_{1}^\T    \\
    \vdots \\
    \bx_N^\T
\end{bmatrix}, \\\\ \quad \bx_i = [ x_i(1), \ldots, x_i(p)]^\T,
\]
 where $p$ is the dimension of the parameter vector. By the ergodicity of the regressors, $ \bA_N /N\convP \bA_0$ according to the ergodic theorem.
 The residuals are defined by
\beq\label{res-1}
\hat u_i = X_i - \bx_i^T \hat\bbe_N,\;\;\;1\leq i \leq N.
\eeq

\begin{theorem} \label{p:lr}
Under Assumptions 2.1--\ref{as-cor}, \ref{a:xu-indep}, and assuming that $\bA_0$ is nonsingular,
Corollaries \ref{basic-cor} and \ref{cor-sq} remain valid for the residuals defined in  \eqref{res-1}.
\end{theorem}

We now turn to the nonlinear regression
\[
X_i = h(\bx_i, \btheta_i) + u_i, \ \ \ 1 \le i \le N,
\]
where $\btheta_i$ are $p$--dimensional parameter vectors,
equal under $H_0$. The unknown
common parameter vector is then $\btheta_0$. It is estimated
by minimizing
\[
L_N(\btheta) = \sum_{i=1}^N (X_i - h(\bx_i, \btheta))^2
\]
over the parameter space $\bTheta$. The following usual assumption
is made.
\begin{assumption} \label{a:theta} The parameter space $\bTheta$ is
a compact subset of ${\mathbb R}^p$  and $\btheta_0$ is its interior point.
\end{assumption}
We impose the following
assumption on the function $h(\cdot, \cdot)$.
\begin{assumption} \label{a:h} The following conditions hold:
\begin{align*}
&\sup_{\btheta \in \bTheta} E h^2(\bx_0, \btheta) < \infty, \\
&\sup_{\btheta \in \bTheta} \left \| \frac{\partial^2}{\partial \btheta^2}
h(\bx_i, \btheta) \right \| \le M(\bx_i), \ \ \ EM(\bx_0) < \infty
\end{align*}
and
\[
E \left \|\frac{\partial}{\partial \btheta} h(\bx_0, \btheta_0) \right \|^2
 < \infty,
\]
\[
E [ h(\bx_0, \btheta_0) - h(\bx_0, \btheta) ]^2 > 0, \ \ \
{\rm if} \ \ \btheta \neq \btheta_0.
\]
\end{assumption}

The conditions formulated in Assumption~\ref{a:h} ensure that
under $H_0$ the differences between the functionals based on the
unobservable errors $u_i$ and the residuals
\beq\label{res-2}
\hat u_i = X_i - h(\bx_i, \hat\btheta)
\eeq
are asymptotically negligible in the sense that they do not affect
the limits in Corollaries \ref{basic-cor} and \ref{cor-sq}.

\begin{theorem} \label{p:nr}
Under Assumptions 2.1--\ref{as-cor} and \ref{a:theta}-\ref{a:h},
Corollaries \ref{basic-cor} and \ref{cor-sq} remain valid for the residuals defined in  \eqref{res-2}.
\end{theorem}

The consistency of the tests in both linear and nonlinear regression
settings can be established in a framework analogous to that
considered in Section~\ref{ss:cm}.

\section{Computation of the limit distributions based on
Karhunen--Lo\'eve expansions} \label{s:comp}
The Karhunen--Lo\'eve expansion yields that
 \begin{align}\label{khl-1}
 \int_0^1\Gamma^2(t)dt=\sum_{i=1}^\infty \lambda_i\xi_i^2,
 \end{align}
 where $\xi_1, \xi_2, \ldots$ are independent and identically
 distributed standard normal random variables, and $\lambda_1\geq
 \lambda_2\geq \ldots$ are the eigenvalues of the operator associated
 with the kernel $C(t,s)$ defined in \eqref{e:Cts}. Specifically,
\begin{equation} \label{khl-2}
 \lambda_i\varphi_i(t)=\int_0^1C(t,s)\varphi_i(s)ds,\;\;\;1\leq i
 <\infty.
\end{equation}
The $\varphi_1, \varphi_2, \ldots$ are orthonormal
 functions, the eigenfunctions of $C(t,s)$. Similarly,
 \begin{align}\label{khl-3}
 \int_0^1\frac{\Gamma^2(t)}{t(1-t)}dt=\sum_{i=1}^\infty \tau_i\xi_i^2,
 \end{align}
 where $\tau_1\geq \tau_2\geq \ldots$ are the eigenvalues of
 $$
 D(t,s)=\frac{C(t,s)}{(t(1-t)s(1-s))^{1/2}}.
 $$

 The eigenvalues $\lambda_i, i\geq 1,$ as well as $\tau_i, i\geq 1, $
 can be estimated  from the sample. This is addressed
 in Sections \ref{sec-est-1} and \ref{sec-est-2}. Using \eqref{khl-1} or
 \eqref{khl-3},  we can obtain critical values for the Cram\'er--von Mises
 and Anderson--Darling statistics by proceeding as follows.
 If $\widehat{C}_N(t,s)$ is an
 $L^2$ consistent estimate of $C(t,s)$, then the empirical eigenvalues
 $\hat{\lambda}_{N,1} \geq \hat{\lambda}_{N,2} \geq
 \hat{\lambda}_{N,3} \geq \ldots$ of $\widehat{C}_N(t,s)$ can be used to
 approximate the sum on the right hand side of \eqref{khl-1}, i.e.\ we
 use the distribution of
\begin{equation}\label{H-def}
\hat{H}_{N,m}=\sum_{\ell=1}^m \hat{\lambda}_{N,i}\xi_i^2,\;\;\mbox{with some large enough}\;m.
 \end{equation}
  The empirical eigenvalues satisfy the integral equation
 \[
  \hat{\lambda}_{N,i}\hat{\varphi}_{N,i}(t)=\int_0^1\widehat{C}_N(t,s)\hat{\varphi}_{N,i}(s)ds,
  \]
  where $\hat{\varphi}_{N,i}(t), i\geq 1,$ are orthonormal
  eigenfunctions. The same method works to approximate the
  distribution in \eqref{khl-3}.

We now turn to the details of the computation of the
$\hat{\lambda}_{N,i}$ and the $\hat{\tau}_{N,i}$, first in the case
of uncorrelated errors, then in the general case of correlated errors.

\medskip
\subsection{Estimation of the eigenvalues in case of  uncorrelated errors}
\label{sec-est-1} ~\\

To illustrate our method,  we first consider uncorrelated observations:
 \begin{assumption}\label{ass-un}
 \begin{displaymath}
 Ee_ie_j=\left\{
 \begin{array}{ll}
 0,\quad &\mbox{if}\;\;i\neq j,
 \vspace{.3cm}\\
 \sigma^2,\quad &\mbox{if}\;\;i=j.
 \end{array}
 \right.
 \end{displaymath}
 \end{assumption}
Assumption \ref{ass-un} holds for a large class of random
 variables, including GARCH--type sequences,  Francq and Zakoian (2010).

 Let
 $$
 \bar{X}_N=\frac{1}{N}\sum_{\ell=1}^NX_\ell
 $$
 and define
 $$
 H_N(u)=\frac{1}{N}\sum_{i=1}^{\lf Nu\rf}(X_i-\bar{X}_N)^2, \quad 0\leq u\leq 1.
 $$
We estimate $C(t,s)$ with
\begin{equation} \label{e:C-HU}
\widehat{C}_N(t,s)=H_N(t\wedge s)-tH_N(s)-sH_N(t)+stH_N(1).
\end{equation}
Let
 $$
 g(u)=\lim_{N\to\infty}\frac{1}{N}\sum_{\ell=1}^{\lf Nu\rf}(\mu_\ell-\bar{\mu})^2.
 $$
 It is clear that $g(u)=0$ for all $0\leq u \leq 1$ under $H_0$.
\begin{theorem}\label{under-th} We assume that Assumptions \ref{as-1}--\ref{as-3}, \ref{ass-un} are satisfied and $\{e_i,-\infty<\infty\}$ is a stationary and ergodic sequence with $Ee_0=0$ and $Ee_0^2=\sigma^2<\infty$.\\
(i)  If $H_0$ holds, then
$$
\int_0^1\!\!\int_0^1\left( \widehat{C}_N(t,s)-C(t,s)  \right)^2dtds=o_P(1).
$$
(ii)  If $H_A$ holds, and in addition Assumption \ref{mu-con} also holds, then
$$
\int_0^1\!\!\int_0^1\left( \widehat{C}_N(t,s)-C^*(t,s)  \right)^2dtds=o_P(1),
$$
where
$$
C^*(t,s)=C(t,s)+g(t\wedge s)-tg(s)-sg(t)+tsg(1).
$$
\end{theorem}

\medskip
We obtain from Theorem \ref{under-th} (see e.g.\ Lemma 2.2 in Horv\'ath and Kokoszka (2012) or  Dunford and Schwartz (1988)) that under $H_0$
\begin{equation}\label{ei-1}
\hat{\lambda}_{N,i}\;\stackrel{P}{\to}\;\lambda_i.
\end{equation}
It is easy to see that $C^*(t,s)$ is a symmetric, non--negative definite function. Let $\lambda_1^*\geq \lambda_2^*\geq\lambda_3^*\geq\ldots$ be the eigenvalues of $C^*$. We conclude that under  $H_A$
\begin{equation}\label{ei-2}
\hat{\lambda}_{N,i}\;\stackrel{P}{\to}\;\lambda_i^*.
\end{equation}

For any $0<\alpha<1$, we define $\hat{c}_{N,m}(\alpha)$ as
$$
\hat{c}_{N,m}(\alpha) =\min\{x: P\{ \hat{H}_{N,m} \geq x\}\leq \alpha\},
$$
where $\hat{H}_{N,m}$ is defined in \eqref{H-def}.  The empirical critical value $\hat{c}_{N,m}(\alpha)$ is asymptotically correct. It follows from Corollary \ref{cor-sq},\eqref{khl-1} and \eqref{ei-1} that under $H_0$
$$
\lim_{m\to \infty}\lim_{N\to\infty}P\left\{\int_0^1Z_N^2(t)dt\geq \hat{c}_{N,m}(\alpha)\right\}=\alpha.
$$
By \eqref{alt-st-1}, \eqref{alt-as} and \eqref{ei-2}, we conclude that under $H_A$
$$
\lim_{N\to\infty}P\left\{\int_0^1Z_N^2(t)dt\geq \hat{c}_{N,m}(\alpha)\right\}=1,\;\;\mbox{for all}\;m\geq 1,
$$
establishing the consistency of  the Cram{\'e}r--von Mises procedure.
The same arguments apply to the Anderson--Darling procedure.

\medskip
 \subsection{Estimation of the eigenvalues in case of  correlated errors}
\label{sec-est-2} ~\\

In the general case of correlated errors, the kernel  $C(t,s)$ is
estimated by
\begin{equation} \label{e:C-HC}
\widetilde{C}_N(t,s)
=\hat{g}_N(t\wedge s)-t\hat{g}_N(s)-s\hat{g}_N(t)+st\hat{g}_N(1),
\;\;0\leq s,t\leq 1,
\end{equation}
where
\[
\hat{g}_N(u)=\hat{g}_{N,\lf Nu\rf},\;0\leq u \leq 1,
\]
 and where $\hat{g}_{N,k}$ is an estimator of the long--run variance
based on the partial sample $X_1, X_2, \ldots X_k$, $ k \le N$.

In the following, we establish the asymptotic validity of the tests,
both under $H_0$ and $H_A$,  for the commonly used  kernel estimator
$\hat{g}_{N,k}$.

For any $1\leq k \leq N$ and $\ell, |\ell|< k$ we define
\begin{displaymath}
\hat{\gamma}_{N;k,\ell}=\hat{\gamma}_{k,\ell}=\left\{
\begin{array}{ll}
\displaystyle \frac{1}{N}\sum_{i=1}^{k-\ell}(X_i-\bar{X}_N)(X_{i+\ell}-\bar{X}_N),\;\;\mbox{if}\;\;0\leq \ell\leq  k-1,
\vspace{.3cm}\\
\displaystyle \frac{1}{N}\sum_{i=-\ell+1}^{k}(X_i-\bar{X}_N)(X_{i+\ell}-\bar{X}_N),\;\;\mbox{if}\;\; -(k-1)\leq \ell <0.
\end{array}
\right.
\end{displaymath}
Let
\[
\hat{g}_{N,k}=\sum_{\ell=-(k-1)}^{k-1}K(\ell/h)\hat{\gamma}_{N;k,\ell},
\]
We assume standard conditions on  the kernel $K$
and window (smoothing) parameter $h$:
\begin{assumption}\label{ass-k} (i) $K(0)=1$
(ii)  $ K(u)=K(-u)\geq 0$
(iii) $K(u)=0$ if $|u|>c$ with some $c>0$
(iv) $K$ is Lipschitz continous on the real line
\end{assumption}
and
\begin{assumption}\label{ass-h}\;\;$h=h(N)\to \infty$ and $h/N\to 0$.
\end{assumption}

The study of the the estimator $\hat{g}_N(t), 0\leq t \leq 1$, requires assumptions on the structure of the innovations $e_i, -\infty<i<\infty$. We assume that the $e_i$'s form a Bernoulli shift:

\begin{assumption}\label{as-ber} (i) $e_i=f(\vare_i, \vare_{i-1}, \ldots)$, where $f$ is a measurable functional and $\vare_i, -\infty<i<\infty$, are independent and identically distributed random variables in some measure space.\\
(ii) $Ee_0=0$ and $E|e_0|^4<\infty$. \\
(iii)
$$(E|e_{i,m}-e_i|^4)^{1/4}=O(m^{-\alpha})\;\;\mbox{ with some }\;\;\alpha>2,
$$
where $e_{i,m}=g(\vare_i, \vare_{i-1}, \ldots ,\vare_{i-m},
 \vare_{i,m, i-m-1}, \vare_{i,m, i-m-2}, \ldots)$, $\vare_{i,m,\ell}, -\infty<i,m, \ell<\infty$,
are independent and identically distributed copies of $\vare_0$.
\end{assumption}
We note that Assumption \ref{as-ber} implies that $e_i, -\infty<i<\infty$, is a stationary sequence and Assumptions \ref{as-3} and \ref{as-cor} are also satisfied (cf.\ Berkes et al.\ (2013) and Jirak (2013)).
To get the exact limit of $\widetilde{C}_N(t,s)$ under $H_A$ we need a further regularity condition:
\begin{assumption}\label{as-sm}
$$
\max_{1\leq \ell\leq ch}\frac{1}{N} \sum_{1\leq i \leq k-\ell}|\mu_{i+\ell}-{\mu}_i|=o(1),
$$
where $c$ is defined in Assumption \ref{ass-k}.
\end{assumption}
It is easy to see that Examples \ref{ex-1}--\ref{ex-3} satisfy Assumption \ref{as-sm}.

\medskip
 \begin{theorem}\label{second-th} We assume that Assumptions \ref{as-1}, \ref{as-2} and \ref{ass-k}--\ref{as-ber} are satisfied.\\
 (i) If $H_0$ holds, then
 \begin{equation}\label{lo-1}
 \int_0^1\!\!\!\int_0^1 (\widetilde{C}_N(t,s)-C(t,s))^2dtds=o_P(1).
 \end{equation}
 (ii) If $H_A$ and, in addition, Assumption \ref{mu-con} hold, then
 \begin{equation}\label{lo-1/2}
 \int_0^1\!\!\!\int_0^1 \widetilde{C}_N^2(t,s)dtds=O_P(h^2),
 \end{equation}
 (iii) If $H_A$ and, in addition, Assumptions \ref{mu-con} and \ref{as-sm} hold, then
 \begin{equation}\label{lo-2}
 \int_0^1\!\!\!\int_0^1 \left(\frac{1}{h}\widetilde{C}_N(t,s)-C^{**}(t,s)\right)^2dtds=o_P(1),
 \end{equation}
 where
 $$
 C^{**}(t,s)=\left( g(t\wedge s)-tg(s)-sg(t)+tsg(1)   \right)\int_{-c}^c K(u)du.
 $$
 \end{theorem}

 \medskip
 Let $\tilde{\lambda}_{N,1}\geq \tilde{\lambda}_{N,2}\geq \ldots$ denote the eigenvalues of $\widetilde{C}_N(t,s)$. It follows from Theorem \ref{second-th}(i), analogously  to \eqref{ei-1}, that
 \begin{equation}\label{co-ei-1}
 \tilde{\lambda}_{N,i}\;\stackrel{P}{\to}\;\lambda_i,
 \end{equation}
where $\lambda_1\geq \lambda_2\geq \ldots $ are the eigenvalues of $C(t,s)$ defined in \eqref{khl-2}. For any $0<\alpha<1$,  we now define
 $\tilde{c}_{N,m}(\alpha)$ as
$$
\tilde{c}_{N,m}(\alpha) =\min\{x: P\{ \tilde{H}_{N,m} \geq x\}\leq \alpha\},
$$
where
$$
\tilde{H}_{N,m}=\sum_{i=1}^m \tilde{\lambda}_{N,i}\xi_i^2,
\;\;\mbox{with some large enough}\;m,
$$
and $\xi_1, \xi_2, \ldots $ are independent standard normal random variables. It follows from Corollary \ref{cor-sq},\eqref{khl-1} and \eqref{co-ei-1} that under $H_0$
$$
\lim_{m\to \infty}\lim_{N\to\infty}P\left\{\int_0^1Z_N^2(t)dt\geq \tilde{c}_{N,m}(\alpha)\right\}=\alpha.
$$
However, the consistency of the procedure is more delicate, since the empirical eigenvalues $\tilde{\lambda}_{N,i}$ might not have a finite limit as $N\to \infty$. Indeed, under Assumption \ref{mu-con} we have that $\tilde{\lambda}_{N,i}/h$ converges in probability to a finite limit. Since
$$
\tilde{\lambda}_{N,i}\tilde{\varphi}_{N,i}(t)=\int_0^1\widetilde{C}_{N}(t,s)\tilde{\varphi}_{N,i}(s)ds,
$$
where the $\tilde{\varphi}_{N,i}(t)$'s are orthonormal functions, we get from \eqref{lo-1/2} via the Cauchy--Schwartz inequality that
\begin{align*}
\tilde{\lambda}_{N,i}^2=\int_0^1(\tilde{\lambda}\tilde{\varphi}_{N,i}(t))^2dt&=\int_0^1\left(\int_0^1\widetilde{C}_{N}(t,s)\tilde{\varphi}_{N,i}(s)ds\right)^2dt\\
&\leq \int_0^1\left(\int_0^1\widetilde{C}_{N,i}^2(s,t)ds \int_0^1  \tilde{\varphi}_{N,i}^2(s)ds\right)dt\\
&=\int_0^1\!\!\!\int_0^1\widetilde{C}_{N,i}^2(s,t)dsdt\\
&=O(h^2).
\end{align*}
Hence under $H_A$ we have that $\tilde{\lambda}_{N,i}=O_P(h)$, implying that for each $m$ and $0<\alpha<1$
$$
\tilde{H}_{N,m}=O_P(h)\;\;\;\mbox{and therefore}\;\;\;\tilde{c}_{N,m}(\alpha)=O_P(h).
$$
Thus Theorem \ref{alt-1} yields
$$
\lim_{N\to\infty}P\left\{\int_0^1Z_N^2(t)dt\geq \tilde{c}_{N,m}(\alpha)\right\}=1\;\;\mbox{for all}\;m\geq 1,
$$
establishing the consistency of  the Cram{\'e}r--von Mises procedure
 in case of correlated observations when Assumption \ref{ass-h} holds.
Similar  arguments apply to the Anderson--Darling procedure.

\medskip
 \section{Simulation study and application to US yields}\label{sec-sim}
The purpose of this section is to assess the finite sample performance of the
 proposed tests. After describing them
 in a systematic manner in Section~\ref{ss:tp}, we explore
 in  Section~\ref{ss:err} their  properties
 using simulated and  real data.

\medskip
\subsection{Test procedures} \label{ss:tp} ~\\

For ease of reference, we begin by listing the test procedures
introduced in this paper and in related work together with convenient
abbreviations.  We also provide their brief descriptions.  The
procedures are based on the following ingredients, which also form
the building blocks of the abbreviations.

{\bf S}tandard vs. {\bf H}eteroskedastic. In the standard approach
we assume that $a(t) =1$, i.e. we do not take the
possible heteroskedasticity of the errors into account.
In the heteroskedastic approach, the function $a(\cdot)$ is estimated
as explained in the previous sections.

{\bf U}ncorrelated vs. {\bf C}orrelated. In the uncorrelated case,
we estimate the eigenvalues as described in Section~\ref{sec-est-1},
i.e. we assume that the observations  are uncorrelated. In the correlated
case, we estimate the eigenvalues as described  in Section~\ref{sec-est-2},
i.e. we assume that the observations  are correlated.

{\bf CM} vs. {\bf AD}. This refers to using either the Cram{\'e}r--von Mises
or the Anderson--Darling functional.

We also consider two methods studied  by  Dalla et al.\ (2015), which
they  denote  {\bf VS}$^*$  and {\bf VS}, but which
we denote {\bf VSU} and {\bf VSC} to emphasize more
clearly that {\bf VSU} assumes uncorrelated errors, while
{\bf VSC} assumes correlated errors. Dalla et al.\ (2015)  also
considered analogous methods based on the KPSS statistic.
 They found that they were not competitive with the VS methods,
so we do not include the KPSS methods in our comparison.

We now list the methods we  study.

{\bf SUCM}
(Standard Uncorrelated errors Cram{\'e}r--von Mises.)
Set
\begin{equation} \label{e:CM}
\widehat T_N = \int_0^1 Z_N^2(t) dt
\end{equation}
and denote by $\hat\sigma^2$ is the sample variance
of the observations $X_i$. Then
\begin{equation} \label{e:CML}
\frac{\widehat T_N}{\hat\sigma^2}
\;\;\stackrel{{\mathcal D}}{\to}\;\;\int_0^1 B^2(t) dt,
\end{equation}
where $B(\cdot)$ is the standard Brownian motion. The distribution
of the right--hand side of \eqref{e:CML} is known,
Shorack and Wellner (1986).

{\bf SCCM} The only difference between this method and SUCM
is that in \eqref{e:CML},  $\hat\sigma^2$ is a consistent estimator of
the long--run variance of the $X_i$.

{\bf HUCM}
(Heteroskedastic Uncorrelated errors Cram{\'e}r--von Mises.)
The test statistic is $\widehat T_N$ defined by \eqref{e:CM}.
Its limit distribution is approximated by the empirical distribution
of the random variable
\[
T(m) = \sum_{i=1}^m \hat\lambda_i \xi_i^2.
\]
The $\xi_j$ are independent standard normal. The $\hat\lambda_i$
satisfy
\[
\hat\lambda_i \hat\varphi_i(t) = \int_0^1 \widehat C (t,s) \hat\varphi_i(s) ds,
\]
where $\widehat{C}$
is given by \eqref{e:C-HU}.
The P--value
is computed as
\[
P= \frac{1}{R} \#\left\{ r: T(m)_r> \widehat T_N \right\},
\]
where $T(m)_r, r=1,2, \ldots, R, $ are independent replications of $T(m)$.

{\bf HCCM} (Heteroskedastic Correlated errors Cram{\'e}r--von Mises.)
Conceptually, the only difference between this method and SUCM is that
$\widehat{C}$ is replaced by $\widetilde{C}$ given by \eqref{e:C-HC}.
We note that $\hat{g}_{N,k}$ is the estimator of the long--run
variance of $X_1, X_2, \ldots, X_k$, and any suitable estimator can be
used. To enhance comparison,
we used the  spectral estimator employed by  Dalla et
al.\ (2015). \footnote{We thank Dr. V. Dalla for making the code available.}

{\bf SUAD}
(Standard Uncorrelated errors Anderson--Darling.)
Set
\begin{equation} \label{e:AD}
\widehat U_N = \int_{1/N}^{1-1/N}\frac{Z_N^2(t)}{t(1-t)} dt
\end{equation}
and denote by $\hat\sigma_N^2$ the sample variance of
the $X_i$. Then
\begin{equation} \label{e:ADL}
\frac{\widehat U_N}{\hat\sigma_N^2}
\;\;\stackrel{{\mathcal D}}{\to}\;\;
\int_0^1 \frac{B^2(t)}{t(1-t)} dt,
\end{equation}
where $B(\cdot)$ is the standard Brownian motion. The distribution
of the right--hand side of \eqref{e:CML} is known,
Shorack and Wellner (1986).

{\bf SCAD} The only difference between this method and SUAD
is that in \eqref{e:ADL},  $\hat\sigma^2$ is a consistent estimator of
the long--run variance of the $X_i$.

{\bf HUAD}
(Heteroskedastic Uncorrelated errors Anderson--Darling.)
The test statistic is $\widehat U_N$ defined by \eqref{e:AD}.
Its limit distribution is approximated by the empirical distribution
of the random variable
\[
U(m) = \sum_{i=1}^m \hat\tau_i \xi_i^2.
\]
The $\xi_j$ are independent standard normal. The $\hat\tau_i$
satisfy
\[
\hat\tau_i \hat\psi_i(t) = \int_{1/N}^{1-1/N}
 \widehat{D}(t,s) \hat\psi_i(s) ds,
\]
where
\[
\widehat D (t,s) = \frac{\widehat{C}(t,s)}{\sqrt{t(1-t)(s(1-s)}},
\]
with $\widehat{C}$ given by \eqref{e:C-HU}. The P--value is computed as
\[
P= \frac{1}{R} \#\left\{ r: U(m)_r> \widehat U_N \right\},
\]
where $U(m)_r, r=1,2, \ldots, R, $ are independent replications of $U(m)$.

{\bf HCAD} (Heteroskedastic Correlated errors Anderson--Darling.)
The only difference between this method and HUAD is that $\widehat{C}$
is replaced by $\widetilde{C}$ given by \eqref{e:C-HC}.
The specific implementation is the same as for the HCCM method.

{\bf VSU} (VS statistic Uncorrelated errors) The test statistic
is
\[
\widehat{V}_N^* = \frac{1}{\hat\sigma^2 N^2}
\sum_{k=1}^N \left ( S_k^\prime - \bar{S^\prime} \right )^2, \ \ \
S_k^\prime = \sum_{i=1}^k (X_i - \bar X),
\]
where $\hat\sigma^2$ is the sample variance of the observations
$X_i, 1 \le i \le N$.
Its  null distribution is approximated by the distribution
of the random variable
\[
V(m) = \sum_{k=1}^m \frac{\chi_k^2(2)}{4\pi^2 k^2},
\]
where  the $\chi_k^2(2)$ are iid chi-square with 2 degrees of freedom.
If $\widehat V_N^*$ is the observed value of the statistic,
then the P--value is computed as
\[
P= \frac{1}{R} \#\left\{ r: V(m)_r> \widehat V_N^* \right\},
\]
where $V(m)_r, r=1,2, \ldots, R, $ are independent replications of $V(m)$.

{\bf VSC} (VS statistic Correlated errors) The only difference between
this method and {\bf VSU} is that in the definition of the test
statistics, say  $\widehat V_N$, $\hat\sigma^2$ is replaced by a consistent
estimator of the long--run variance of the observations $X_i, 1 \le i
\le N$, i.e. by $\hat g_{N,N}$ in the notation of Section~\ref{sec-est-2}.

We emphasize that, in contrast to the  H--methods introduced in
this paper, the common asymptotic distribution of the statistics
$\widehat{V}_N^*$ and $\widehat V_N$ does not depend on the data.
These statistics do not directly take possible heteroskedasticity into
account; their applicability is based on  empirically and
theoretically established relative  insensitivity to
heteroskedastic  errors.

\medskip
\subsection{Empirical rejection rates and application to US yields}
\label{ss:err}~\\

We analyzed the size and power of the tests for all models
considered by Dalla et al.\ (2015). Regarding the empirical
size, our HU tests have similar size as the VSU test;
the HC tests have size similar to the VSC test. Generally, the
differences in empirical size within these two categories of tests
are within the standard error of the rejection rates. For illustration,
Table~\ref{tb:size2} provides selected results in case of correlated
and heteroskedastic errors, the most general case. With  prior knowledge
that the errors are uncorrelated, the U--methods can be expected
to have correct size only if the observations have the form $X_i = a_i R_i$.
The $R_i$ are realizations of a GARCH process. Without any prior knowledge
about correlation and heteroskedasticity of the errors, only methods
HCCM, HCAD and VSC should be applied. In most cases, there is
no clear advantage of any of these methods over the other. If the
observations have heavy tails, the case of $X_i= a_i(R_i^2 - E R_0^2)$
with $\beta=0.7$, the VSC method overrejects, the empirical size is
over 8\% at the nominal size of 5\%. Generally, our HC methods
tend to have size slightly smaller than the nominal size, the VSC method
a somewhat larger size.

\begin{table}[!ht]
\tiny
\centering
\begin{tabular}{lcccccccccccc}
\toprule
&$\mathbf{N=128}$ & SUCM & SCCM & HUCM & HCCM & SUAD & SCAD & HUAD & HCAD & VSU & VSC\\
\midrule
$X_i = a_{i1}R_i$ & \multirow{8}{*}{$\beta = 0.5$}  & 7.1 & 6.9 & 5.1 & 3.2 & 7.5 & 10.1 & 4.6 & 3.3 & 6.6 & 6.8 \\
$X_i = a_{i1}(|R_i| - E|R_0|)$ & & 24.9 & 6.8 & 20.7 & 4.7 & 26.4 & 7.6 & 19.8 & 4.2 & 27.7 & 3.8\\
$X_i = a_{i1}(R_i^2 - ER_0^2)$ & & 28.2 & 8.9 & 29.2 & 5.8 & 28.5 & 9.4 & 29.3 & 5.3 & 34.0 & 4.5\\
\cmidrule{1-1} \cmidrule{3-12}
$X_i = a_{i2}R_i$ & & 4.5 & 5.1 & 4.6 & 4.2 & 5.2 & 8.1 & 4.9 & 4.0 & 5.2 & 5.8 \\
$X_i = a_{i2}(|R_i| - E|R_0|)$ & & 20.5 & 7.6 & 20.0 & 4.7 & 21.1 & 5.6 & 20.8 & 4.7 & 29.0 & 5.3 \\
$X_i = a_{i2}(R_i^2 - ER_0^2)$ & & 25.1 & 7.9 & 24.6 & 4.3 & 23.8 & 6.4 & 24.4 & 4.6 & 36.2 & 5.8\\
\midrule
$X_i = a_{i1}R_i$ & \multirow{8}{*}{$\beta = 0.7$} & 7.7 & 7.1 & 5.2 & 4.1 & 8.5 & 10.4 & 4.6 & 4.0 & 6.8 & 5.7\\
$X_i = a_{i1}(|R_i| - E|R_0|)$ & & 49.9 & 10.4 & 43.5 & 5.3 & 52.2 & 9.0 & 43.2 & 4.2 & 58.6 & 5.4\\
$X_i = a_{i1}(R_i^2 - ER_0^2)$ & & 57.2 & 13.7 & 48.3 & 7.8 & 58.7 & 13.3 & 47.5 & 6.1 & 62.8 & 8.0\\
\cmidrule{1-1} \cmidrule{3-12}
$X_i = a_{i2}R_i$ & & 5.9 & 6.1 & 3.9 & 4.3 & 5.8 & 6.9 & 3.9 & 4.1 & 5.1 & 4.6 \\
$X_i = a_{i2}(|R_i| - E|R_0|)$ & & 46.5 & 9.7 & 45.6 & 5.0 & 47.3 & 8.5 & 45.6 & 4.5 & 58.0 & 7.3\\
$X_i = a_{i2}(R_i^2 - ER_0^2)$ & & 49.6 & 10.9 & 47.8 & 4.8 & 49.3 & 10.0 & 46.9 & 4.8 & 61.3 & 8.9\\
\bottomrule
\toprule
&$\mathbf{N=512}$ & SUCM & SCCM & HUCM & HCCM & SUAD & SCAD & HUAD & HCAD & VSU & VSC\\
\midrule
$X_i = a_{i1}R_i$ & \multirow{8}{*}{$\beta = 0.5$}  & 6.8 & 7.6 & 3.9 & 5.1 & 8.0 & 7.5 & 4.1 & 5.4 & 5.9 & 5.0\\
$X_i = a_{i1}(|R_i| - E|R_0|)$ & & 26.7 & 9.5 & 23.7 & 6.3 & 35.3 & 9.4 & 25.6 & 6.4 & 35.4 & 6.9 \\
$X_i = a_{i1}(R_i^2 - ER_0^2)$ & & 32.4 & 10.4 & 25.3 & 4.1 & 37.5 & 12.3 & 26.6 & 4.3 & 39.4 & 6.0\\
\cmidrule{1-1} \cmidrule{3-12}
$X_i = a_{i2}R_i$ & & 5.9 & 6.2 & 4.9 & 4.6 & 4.9 & 5.9 & 5.4 & 5.0 & 4.7 & 5.3\\
$X_i = a_{i2}(|R_i| - E|R_0|)$ & & 26.3 & 6.0 & 23.4 & 4.6 & 29.2 & 7.4 & 26.1 & 5.4 & 37.1 & 7.1\\
$X_i = a_{i2}(R_i^2 - ER_0^2)$ & & 28.7 & 6.7 & 29.3 & 5.3 & 32.4 & 7.2 & 31.1 & 5.6 & 43.5 & 6.7\\
\midrule
$X_i = a_{i1}R_i$ & \multirow{8}{*}{$\beta = 0.7$} & 6.1 & 5.5 & 5.2 & 4.2 & 8.2 & 7.9 & 5.1 & 5.4 & 6.5 & 5.4\\
$X_i = a_{i1}(|R_i| - E|R_0|)$ & & 61.6 & 9.9 & 52.8 & 6.1 & 67.7 & 9.9 & 56.2 & 4.5 & 75.9 & 5.7\\
$X_i = a_{i1}(R_i^2 - ER_0^2)$ & & 65.4 & 14.6 & 60.3 & 6.9 & 71.1 & 15.6 & 63.4 & 5.9 & 78.6 & 8.3\\
\cmidrule{1-1} \cmidrule{3-12}
$X_i = a_{i2}R_i$ & & 6.5 & 5.9 & 4.9 & 5.2 & 6.7 & 5.9 & 5.2 & 5.4 & 5.6 & 5.5\\
$X_i = a_{i2}(|R_i| - E|R_0|)$ & & 59.2 & 8.5 & 55.9 & 4.4 & 63.7 & 7.6 & 62.1 & 4.5 & 75.0 & 5.7\\
$X_i = a_{i2}(R_i^2 - ER_0^2)$ & & 64.3 & 9.4 & 60.7 & 5.2 & 68.8 & 8.9 & 65.1 & 5.7 & 80.0 & 8.6\\
\bottomrule
\end{tabular}
\medskip
\caption{Empirical sizes under nonlinear dependence (at 5\% nominal level). $a_{i1}=i/ 2N$; $a_{i2}=0.25\mathbf{I}\{0\le i\le 0.5N\}+0.5\mathbf{I}\{0.5N< i\le N\}$;
The $R_i$ are GARCH(1,1) processes with $\omega = 10^{-6},  \alpha = 0.2, \beta = 0.5$, alternatively $\beta= 0.7$. }
\label{tb:size2}
\end{table}

\begin{table}[!ht]
\tiny
\centering
\begin{tabular}{lcccccc}
\toprule
& SUCM & SCCM & HUCM & HCCM & VSU & VSC\\
\midrule
\multicolumn{7}{c}{$\mathbf{N=128}$}\\
\midrule
$X_i = \mu_{i}+\sin(\pi i/N)Y_i$; \  $Y_i = \ar(1), \ \rho = 0.5$ & 88.1 & 28.9 & 95.0 & 38.8 & 84.2 & 17.4\\
$X_i = \mu_{i}+\sin(\pi i/N)Y_i$; \ $Y_i\sim N(0,1)$ & 98.2 & 31.6 & 100.0 & 34.6 & 89.3  & 17.2\\
$X_i = \mu_{i}+\sin(\pi i/N)Y_i$; \ $Y_i = \garch(1,1)\ \alpha=0.2,\ \beta = 0.5$ & 100.0 & 3.0 & 100.0 & 88.8 & 100.0  & 0.0\\
$X_i = \mu_{i}+\sin(\pi i/N)Y_i$; \ $Y_i = \garch(1,1)\ \alpha=0.2,\ \beta = 0.7$ & 100.0 & 2.3 & 100.0 & 71.3 & 100.0  & 0.1\\
$X_i = \mu_{i}+a_{i1}Y_i$; \  $Y_i = \ar(1), \ \rho = 0.5$ & 100.0 & 60.7 & 100.0 & 53.0 & 100.0 & 22.1\\
$X_i = \mu_{i}+a_{i1}Y_i$; \ $Y_i\sim N(0,1)$ & 100.0 & 11.5 & 100.0 & 60.7 & 100.0 & 0.0 \\
$X_i = \mu_{i}+a_{i1}Y_i$; \ $Y_i = \garch(1,1)\ \alpha=0.2,\ \beta = 0.5$ & 100.0 & 0.0 & 100.0 & 99.9 & 100.0 & 0.0\\
$X_i = \mu_{i}+a_{i1}Y_i$; \ $Y_i = \garch(1,1)\ \alpha=0.2,\ \beta = 0.7$ & 100.0 & 0.0 & 100.0 & 100.0 & 100.0 & 0.0 \\
$X_i = \mu_{i}+a_{i2}Y_i$; \  $Y_i = \ar(1), \ \rho = 0.5$ & 98.9 & 65.7 & 98.8 & 48.4 & 99.1 & 27.7\\
$X_i = \mu_{i}+a_{i2}Y_i$; \ $Y_i\sim N(0,1)$ & 100.0 & 16.6 & 100.0 & 37.6 & 100.0 & 1.2\\
$X_i = \mu_{i}+a_{i2}Y_i$; \ $Y_i = \garch(1,1)\ \alpha=0.2,\ \beta = 0.5$ & 100.0 & 0.0 & 100.0 & 99.9 & 100.0 & 0.0\\
$X_i = \mu_{i}+a_{i2}Y_i$; \ $Y_i = \garch(1,1)\ \alpha=0.2,\ \beta = 0.7$ & 100.0 & 0.0 & 100.0 & 100.0 & 100.0 & 0.0\\
$X_i = \mu_{i} + a_{i3}Y_i$; \  $Y_i = \ar(1), \ \rho = 0.5$ & 41.1 & 8.5 & 40.0 & 7.4 & 54.2 & 6.3\\
$X_i = \mu_{i} + a_{i3}Y_i$; \ $Y_i\sim N(0,1)$ & 14.8 & 13.9 & 13.3 & 8.3 & 12.1 & 9.7 \\
$X_i = \mu_{i} + a_{i4}Y_i$; \  $Y_i = \ar(1), \ \rho = 0.5$ & 80.8 & 36.3 & 80.6 & 7.0 & 84.2 & 19.0\\
$X_i = \mu_{i} + a_{i4}Y_i$; \ $Y_i\sim N(0,1)$ & 94.3 & 61.1 & 94.0 & 6.1 & 88.2 & 41.8 \\
\midrule
\multicolumn{7}{c}{$\mathbf{N=512}$}\\
\midrule
$X_i = \mu_{i}+\sin(\pi i/N)Y_i$; \  $Y_i = \ar(1), \ \rho = 0.5$ & 100.0 & 98.7 & 100.0 & 99.8 & 100.0 & 90.2\\
$X_i = \mu_{i}+\sin(\pi i/N)Y_i$; \ $Y_i\sim N(0,1)$ & 100.0 & 100.0 & 100.0 & 100.0 & 100.0 & 98.2\\
$X_i = \mu_{i}+\sin(\pi i/N)Y_i$; \ $Y_i = \garch(1,1)\ \alpha=0.2,\ \beta = 0.5$ & 100.0 & 100.0 & 100.0 & 100.0 & 100.0  & 100.0\\
$X_i = \mu_{i}+\sin(\pi i/N)Y_i$; \ $Y_i = \garch(1,1)\ \alpha=0.2,\ \beta = 0.7$ & 100.0 & 100.0 & 100.0 & 100.0 & 100.0  & 100.0\\
$X_i = \mu_{i}+a_{i1}Y_i$; \  $Y_i = \ar(1), \ \rho = 0.5$ & 100.0 & 100.0 & 100.0 & 100.0 & 100.0 & 99.9\\
$X_i = \mu_{i}+a_{i1}Y_i$; \ $Y_i\sim N(0,1)$ & 100.0 & 100.0 & 100.0 & 100.0 & 100.0 & 100.0 \\
$X_i = \mu_{i}+a_{i1}Y_i$; \ $Y_i = \garch(1,1)\ \alpha=0.2,\ \beta = 0.5$ & 100.0 & 100.0 & 100.0 & 100.0 & 100.0 & 100.0\\
$X_i = \mu_{i}+a_{i1}Y_i$; \ $Y_i = \garch(1,1)\ \alpha=0.2,\ \beta = 0.7$ & 100.0 & 100.0 & 100.0 & 100.0 & 100.0 & 100.0\\
$X_i = \mu_{i}+a_{i2}Y_i$; \  $Y_i = \ar(1), \ \rho = 0.5$ & 100.0 & 100.0 & 100.0 & 100.0 & 100.0 & 99.7\\
$X_i = \mu_{i}+a_{i2}Y_i$; \ $Y_i\sim N(0,1)$ & 100.0 & 100.0 & 100.0 & 100.0 & 100.0 & 100.0\\
$X_i = \mu_{i}+a_{i2}Y_i$; \ $Y_i = \garch(1,1)\ \alpha=0.2,\ \beta = 0.5$ & 100.0 & 100.0 & 100.0 & 100.0 & 100.0 & 100.0\\
$X_i = \mu_{i}+a_{i2}Y_i$; \ $Y_i = \garch(1,1)\ \alpha=0.2,\ \beta = 0.7$ & 100.0 & 100.0 & 100.0 & 100.0 & 100.0 & 100.0\\
$X_i = \mu_{i} + a_{i3}Y_i$; \  $Y_i = \ar(1), \ \rho = 0.5$ & 53.5 & 16.4 & 51.7 & 65.6 & 64.7 & 13.2\\
$X_i = \mu_{i} + a_{i3}Y_i$; \ $Y_i\sim N(0,1)$ & 38.9 & 38.2 & 37.3 & 80.2 & 32.9 & 31.1 \\
$X_i = \mu_{i} + a_{i4}Y_i$; \  $Y_i = \ar(1), \ \rho = 0.5$ & 99.9 & 92.6 & 99.8 & 39.6 & 99.6 & 86.1\\
$X_i = \mu_{i} + a_{i4}Y_i$; \ $Y_i\sim N(0,1)$ & 100.0 & 99.9 & 100.0 & 93.1 & 100.0 & 98.6 \\
\bottomrule
\end{tabular}
\medskip
\caption{Empirical power under nonlinear dependence (at 5\% nominal level). $\mu_{i} = 0.5\mathbf{I}\{i>0.5N\}$, $a_{i1} = i/ 2N$; $a_{i2} = 0.25\mathbf{I}\{0\le i\le 0.5N\}+0.5\mathbf{I}\{0.5N< i\le N\}$, $a_{i3} = 1 + 3\mathbf{I}\{i > 0.5N\}$, $a_{i4} = 1 - 0.75\mathbf{I}\{i > 0.5N\}$}
\label{tb:power}
\end{table}

Despite the oversized rejection rates under $H_0$, the VSC method
often  has smaller power the our HC methods. This illustrated in
Table~\ref{tb:power}. Only the CM methods are included, the
rejection rates for the AD methods are very similar. It is however possible
to find cases in which the VSC method dominates our HC methods.
In our simulations, this happens if the variance or the errors drops a lot.
In the cases of $a_{i4}$ the variance of the errors drops from
1 to 1/16  in the second half of the sample. The HC methods apparently
``keep'' the larger estimates of the variances based on initial realizations
$X_1, X_2, \ldots X_k$. These much  larger initial variances suppress the
values of the  HC test  statistics,  resulting in smaller power.

To shed more light on the power behavior of these tests, we apply them
to time series of yieds on US treasury bills,  which are shown in
Figure~\ref{fig:1}. There is an obvious drop in the yields, $H_A$ is
true, which
is accompanied by a drop in variance. Such a scenario might appear to
favor the VSC method. However, as Table~\ref{tb:y-1} shows, it
does not detect a fairly obvious change. Our methods detect
the change in 3 month yields with statistical significance, and a change
in 12 month yields with borderline significance (P--values about 5 percent).

\begin{figure}[!hbtp]
 \centering
\fboxsep=0pt
\noindent%
\includegraphics[height= 2.75 in]{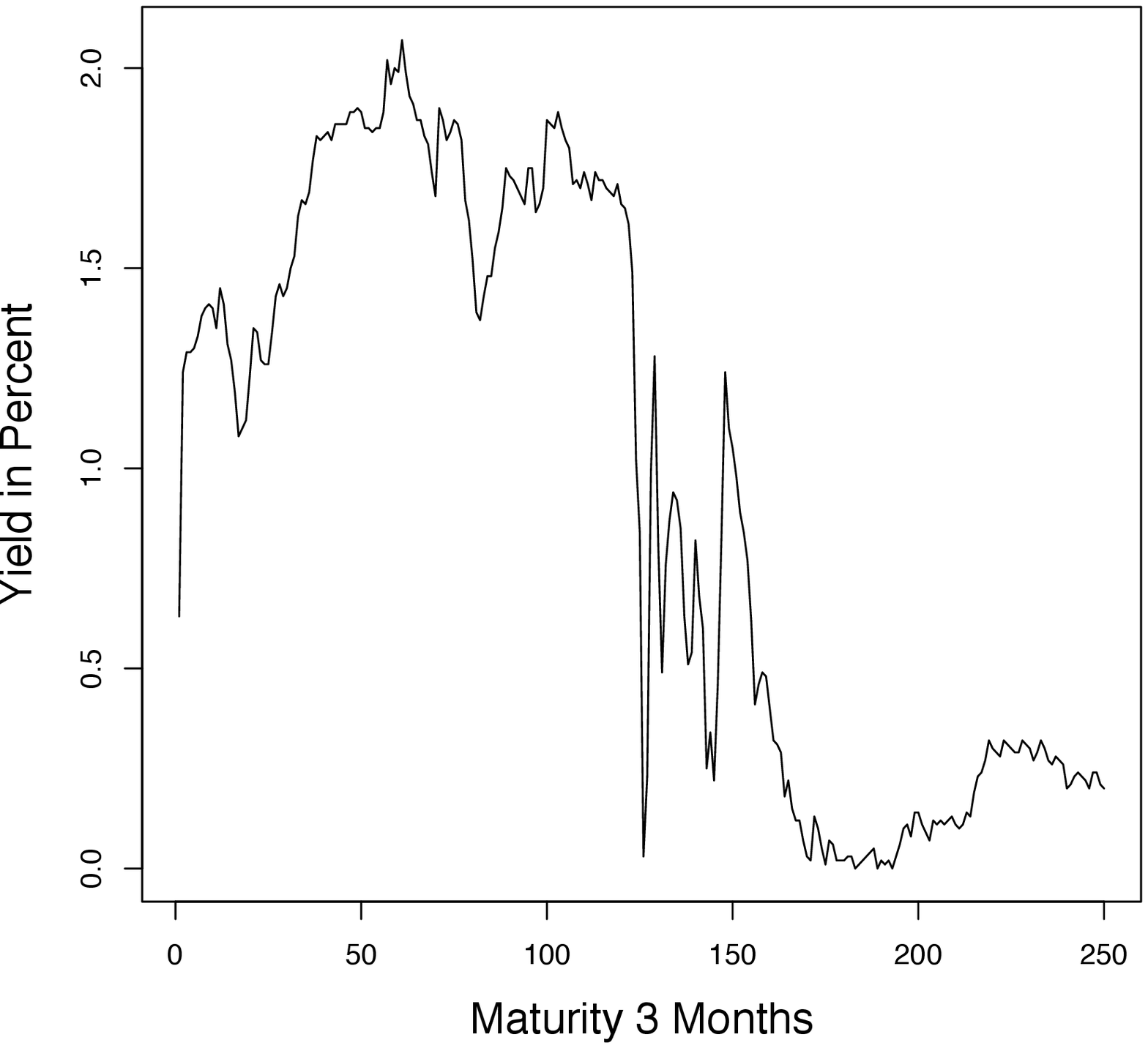}
\includegraphics[height= 2.75 in]{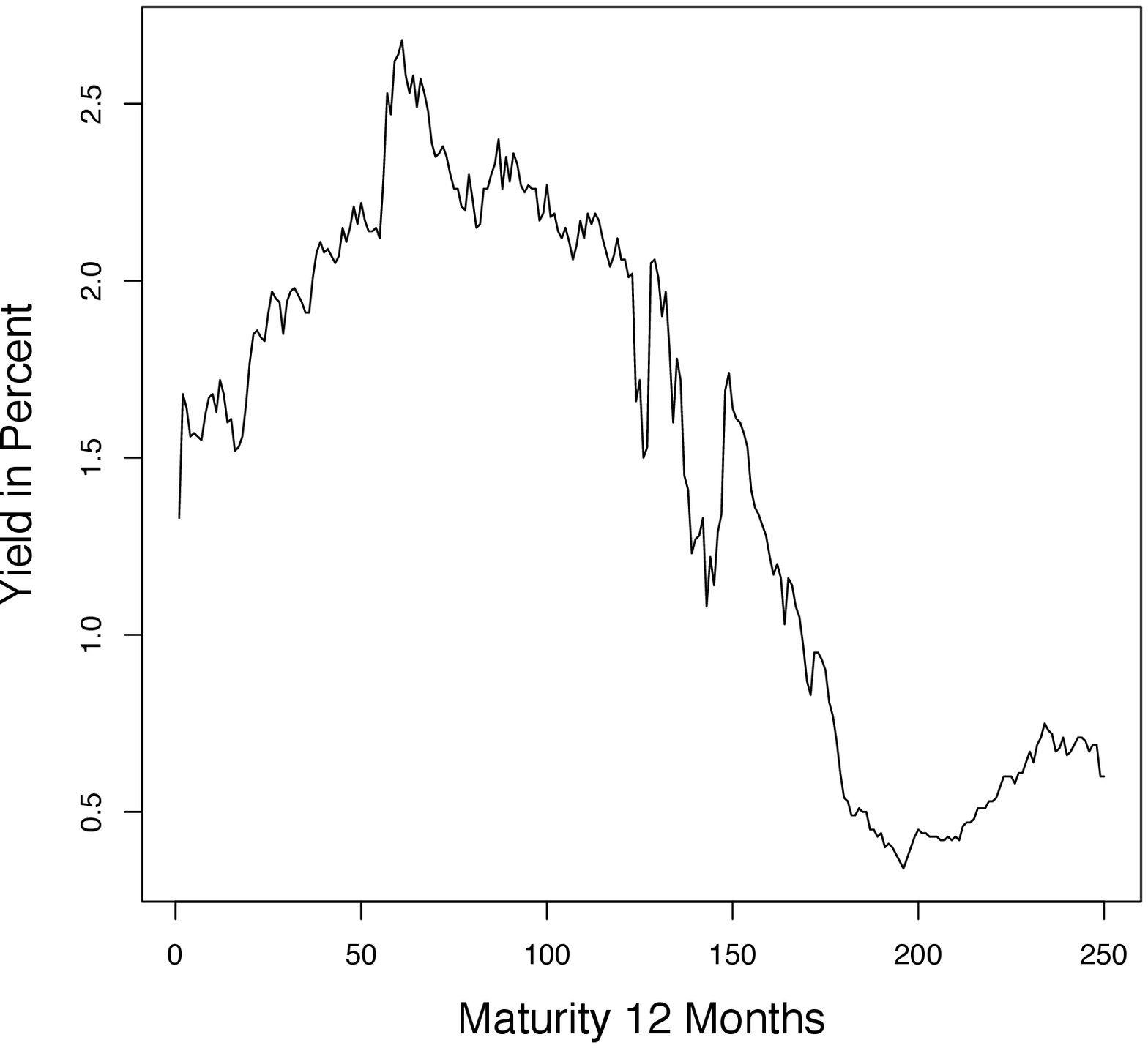}
\caption{Yields over a one year period
on US treasury bills with  maturities of 3 and 12 months.}
\label{fig:1}
\end{figure}

\begin{table}[htb]
\footnotesize
\begin{center}
\begin{tabular}{lcccccc}
\toprule
& HCCM & HCAD & VSC\\
\midrule
3 Month & 0.0243 & 0.0347 & 0.1072\\
12 Month& 0.0545 & 0.0503 & 0.1073\\
\bottomrule
\end{tabular}
\end{center}
\medskip
\caption{P--values for time series of yields in Figure~\ref{fig:1}.}
\label{tb:y-1}
\end{table}

\clearpage

\section{Proofs of the asymptotic results} \label{sec-basicpr}

\subsection{Proofs of the results of Section~\ref{s:alr}} ~\\

{\it Proof of Theorem \ref{basic}.}
Let $S(k)=\sum_{\ell=1}^k e_i$ and $S(0)=0$. By the Abel's summation formula we have
 $$
 \sum_{\ell=1}^k u_\ell=\sum_{\ell=1}^ka(\ell/N)e_\ell=a(k/N)S(k)-\sum_{\ell=1}^{k-1}S(\ell)(a((\ell+1)/N)-a(\ell/N)),\;\;1\leq k \leq N.
 $$
 Since under Assumption \ref{as-3},
 $$
 N^{-1/2}S(\lf Nt\rf)\;\;\stackrel{{\mathcal D}[0,1]}{\longrightarrow}\;\;\sigma W(t),
 $$
 by the Skorohod--Dudley--Wichura representation (cf.\ Shorack and Wellner (1986, p.\ 47)) we can define Wiener processes $W_N(t)$ such that
 $$
 \max_{1\leq k \leq N}\left|S(k)-\sigma W_N(k)\right|=o_P(N^{1/2}).
 $$
Hence, by Assumption \ref{as-2},
\begin{align*}
\max_{1\leq k \leq N}&\left|\sum_{\ell=1}^{k}u_\ell-\sigma\left(a(k/N) W_N(k)-\sum_{\ell=1}^{k}W_N(\ell)(a((\ell+1)/N)-a(\ell/N))  \right)\right|\\
&\leq \max_{1\leq k\leq N}a(k/N)\left|S(k)-\sigma W_N(k)\right|\\
&\hspace{1cm}+\max_{1\leq k\leq N}\left|\sum_{\ell=1}^{k-1} (S(\ell)-\sigma W_N(\ell))(a((\ell+1)/N)-a(\ell/N))\right|\\
&=o_P(N^{1/2})\sup_{0\leq t \leq 1}a(t)+o_P(N^{1/2})\sum_{\ell=1}^{N-1}|a((\ell+1)/N)-a(\ell/N)|\\
&=o_P(N^{1/2}).
\end{align*}
By the Jordan decomposition theorem (cf.\ Hewitt and Stromberg (1969, p.\ 266)), there are two nondecreasing functions such that $a(x)=a_1(x)-a_2(x).$ Focusing on the function $a_1$,  we can write
\begin{align*}
\sum_{\ell=1}^{k-1}&W_N(\ell)(a_1((\ell+1)/N)-a_1(\ell/N))\\
&=\sum_{\ell=1}^{k-1}W_N(\ell)\int_\ell^{\ell+1}da_1( x /N )\\
&=\int_0^k W_N(x)da_1( x/N) +\sum_{\ell=1}^{k-1}\int_\ell^{\ell+1}(W_N(\ell)-W_N(x))da_1( x/N)
\end{align*}

By the modulus of continuity of the Wiener process (cf.\ Lemma 1.2.1 of Cs\"org\H{o} and R\'ev\'esz (1981, p.\ 29)) we have that
$$
\sup_{0\leq u\leq N}\sup_{0\leq x \leq 1}|W_N(u)-W_N(u+x)|=O_P((\log N)^{1/2}).
$$
Integration by parts gives
$$
W_N(k)a_1(k/N)-\int_0^k W_N(x)da_1( x /N)=\int_0^k a_1(x/N)dW_N(x)
$$
and therefore
\begin{align*}
&\left|a_1(k/N) W_N(k)-\sum_{\ell=1}^{k}W_N(\ell)(a_1((\ell+1)/N)-a_1(\ell/N)) - \int_0^k a_1(x/N)dW_N(x) \right|\\
&\hspace{3cm}=O_P((\log N)^{1/2}).
\end{align*}
Similarly,
\begin{align*}
&\left|a_2(k/N) W_N(k)-\sum_{\ell=1}^{k}W_N(\ell)(a_2((\ell+1)/N)-a_2(\ell/N)) - \int_0^k a_2(x/N)dW_N(x) \right|\\
&\hspace{3cm}=O_P((\log N)^{1/2}),
\end{align*}
resulting in
\begin{align*}
\max_{1\leq k\leq N}\left|\sum_{\ell=1}^ku_\ell-\sigma \int_0^k a(x/N)dW_N(x)\right|=o_P(N^{1/2}).
\end{align*}
Let
$$
U_N(t)=\int_0^t a(x/N)dW_N(x), 0\leq t \leq N.
$$
It is easy to see that
$$
\left\{U_N(t), 0\leq t \leq N\right\}\stackrel{{\mathcal D}}{=}\left\{ W\left(\int_0^t a^2(x/N)dx\right), \quad 0\leq t \leq N\right\},
$$
where $W(\cdot)$ is a Wiener process. Next we note that
\begin{align*}
\max_{0\leq k\leq N-1}\sup_{0\leq v\leq 1}\left|W\left(\int_0^k a^2(x/N)dx\right)-W\left(\int_0^{k+v} a^2(x/N)dx\right)\right|=O_P((\log N)^{1/2}),
\end{align*}
since
$$
\max_{0\leq k\leq N-1}\sup_{0\leq v\leq 1}\int_k^{k+v} a^2(x/N)dx\leq 4(a_1^2(1)+a_2^2(1)).
$$
Thus,
we conclude that
$$
\sup_{0\leq t \leq 1}\left|N^{-1/2}\sum_{\ell=1}^{\lf Nt\rf}u_{\ell}-\frac{\sigma}{N^{1/2}} \int_0^{Nt}a(x/N)dW_N(x)\right|=o_P(1).
$$
Since
$$
\int_0^{Nt}a(x/N)dW_N(x)=\int_0^t a(z)dW_N(Nz),
$$
by the scale transformation of the Wiener process we get that
$$
\left\{ N^{-1/2}\int_0^{Nt}a(x/N)dW_N(x), 0\leq t \leq 1\right\} \stackrel{{\mathcal D}}{=}\left\{ \int_0^t a(x)dW(x), 0\leq t \leq 1\right\}.
$$
Computing the covariance function, one can easily verify that
$$
\biggl\{\sigma \int_0^t a(x)dW(x), 0\leq t \leq 1\biggl\}\stackrel{{\mathcal D}}{=}\biggl\{W(b(t)),  0\leq t \leq 1\biggl\},
$$
completing  the the proof of Lemma \ref{basic}. \qed

\medskip
 \noindent
 {\it Proof of Corollary \ref{basic-cor}.} It follows immediately from Theorem \ref{basic}. \qed

 \medskip
 \noindent
 {\it Proof of Corollary \ref{cor-sq}.} The convergence in distribution in \eqref{cv} is an immediate convergence of the continuous mapping theorem and Corollary \ref{basic-cor} (cf.\ Billingsley (1968)).\\
 The result in Corollary \ref{basic-cor} can be restated by an
 application of the Skorohod--Dudley--Wichura representation (cf.\
 Shorack and Wellner (1986, p.\ 47)) that there are Gaussian processes
 $\Gamma_N(t)$ distributed as $\Gamma(t)$ for each $N$ such that
 \begin{equation}\label{ga-co} \sup_{0\leq t \leq 1}|Z_N(t)-\Gamma_N(t)|=o_P(1).
 \end{equation} Let $0<\delta<1/2$. We write by the Cauchy--Schwartz inequality
 that
 \begin{align*}
 &\left| \int_{1/N}^{1-1/N}\frac{Z_N^2(t)}{t(1-t)}dt- \int_{1/N}^{1-1/N} \frac{\Gamma^2_N(t)}{t(1-t)}dt \right|\\
 &\hspace{1cm}\leq \int_{1/N}^{1-1/N}\frac{|Z_N(t)-\Gamma_N(t)|}{(t(1-t))^{1/2-\delta}}
 \frac{|Z_N(t)|+|\Gamma_N(t)|}{t(1-t)^{1/2+\delta}}dt\\
 &\hspace{1cm}\leq 2\left(\int_{1/N}^{1-1/N}\frac{(Z_N(t)-\Gamma_N(t))^2}{(t(1-t))^{1-2\delta}}dt\right)^{1/2}\\
 &\hspace{1.99cm}\times\left(  \int_{1/N}^{1-1/N}\frac{Z_N^2(t)}{(t(1-t))^{1+2\delta}} dt
 + \int_{1/N}^{1-1/N}\frac{\Gamma_N^2(t)}{(t(1-t))^{1+2\delta}} dt \right)^{1/2}.
 \end{align*}
 It follows from \eqref{ga-co} that
 \begin{align*}
 \int_{1/N}^{1-1/N}\frac{(Z_N(t)-\Gamma_N(t))^2}{(t(1-t))^{1-2\delta}}dt\leq \sup_{0\leq t \leq 1}|Z_N(t)-\Gamma_N(t)|\int_0^1(t(1-t))^{-1+2\delta}dt=o_P(1).
 \end{align*}
Next we show that
\begin{align}\label{mom-1}
\int_{1/N}^{1-1/N}\frac{Z_N^2(t)}{(t(1-t))^{1+2\delta}} dt=O_P(1)\;\;\;\mbox{and}\;\; \int_{1/N}^{1-1/N}\frac{\Gamma_N^2(t)}{(t(1-t))^{1+2\delta}} dt=O_P(1).
\end{align}
We note that $E\Gamma_N^2(t)=b(t)-2tb(t)+t^2b(1)=(b(t)-b(1))(1-2t)+(1-t)^2b(1)$ and therefore
\begin{equation}\label{tail-1}
E\Gamma_N^2(t)\leq c_1t(1-t)\;\;\mbox{with some constant}\;\;c_1
\end{equation}
 resulting in
$$
E\int_{1/N}^{1-1/N}\frac{\Gamma_N^2(t)}{(t(1-t))^{1+2\delta}} dt\leq c_1\int_0^1(t(1-t))^{-\delta}dt,
$$
which proves the second half of \eqref{mom-1}. Using Assumption \ref{as-2} and $|Ee_ie_k|=O((|k-i|+1)^{-\kappa})$
we get
\begin{align*}
E\left(\sum_{\ell=k}^m u_\ell\right)^2=\sum_{\ell=k}^m\sum_{\ell'=k}^m a(\ell/N)a(\ell/N)Ee_{\ell}e_{\ell'}\leq c_2\sum_{\ell=k}^m\sum_{\ell'=k}^m |Ee_{\ell}e_{\ell'}|\leq c_3(m-k)
\end{align*}
with some constants $c_2$ and $c_3$. Hence $EZ^2_N(t)\leq c_4t(1-t)$ for all $1/N\leq t \leq 1-1/N$ which implies immediately that
$$
E\int_{1/N}^{1-1/N}\frac{Z_N^2(t)}{(t(1-t))^{1+2\delta}} dt=O(1),
$$
completing the proof of the first part of \eqref{mom-1} via Markov's inequality. We obtain from \eqref{tail-1} that
$$
\int_0^{1/N}\frac{\Gamma_N^2(t)}{t(1-t)}dt=o_P(1)\;\;\mbox{and}\;\;\int_{1-1/N}^1\frac{\Gamma_N^2(t)}{t(1-t)}dt=o_P(1),
$$
which yields
$$
\int^{1-1/N}_{1/N}\frac{\Gamma_N^2(t)}{t(1-t)}dt\;\;\stackrel{{\mathcal D}}{\to}\;\;\int_0^{1}\frac{\Gamma^2(t)}{t(1-t)}dt,
$$
since the distribution of $\Gamma_N(t)$ does not depend in $N$.
 \qed

\medskip
\noindent
{\it Proof of Theorem \ref{alt-1}.} It follows from the definition of $X_i$ that
$$
Z_N(t)=N^{-1/2}\left(\sum_{\ell=1}^{\lf Nt \rf}u_\ell-\frac{\lf Nt\rf}{N}\sum_{\ell=1}^Nu_\ell\right) +N^{-1/2}d_{\lf Nt\rf}.
$$
By Theorem \ref{basic} we have that
$$
\sup_{0\leq t \leq 1}N^{-1/2}\left| \sum_{\ell=1}^{\lf Nt \rf}u_\ell-\frac{\lf Nt\rf}{N}\sum_{\ell=1}^Nu_\ell   \right| =O_P(1)
$$
and by Assumption \ref{mu-con}
$$
\sup_{0\leq t \leq 1}\left| N^{-1/2}\left( \sum_{\ell=1}^{\lf Nt \rf}u_\ell-\frac{\lf Nt\rf}{N}\sum_{\ell=1}^Nu_\ell\right)  N^{-1/2}d_{\lf Nt\rf } \right|
=O_P(N^{1/2}).
$$
Hence \eqref{alt-st-1} follows from the definition of $d(u)$ and Assumption \ref{mu-con} via the Lebesgue dominated convergence theorem (cf.\ Hewitt and Stromberg (1969, p.\ 172)).\\
Similar arguments can be used to prove \eqref{alt-st-1} and therefore the details are omitted.
\qed

\medskip
\noindent
{\it Proof of Theorem \ref{p:lr}.}
We have the  usual representation
\[
\hat\bbe_N - \bbe_0 = \bA_N^{-1} \bZ_N,
\]
where $\bbe_0$ is  the true parameter vector under $H_0$ and
$\bZ_N = [ Z_N(1), \ldots, Z_N(p) ]^\T$ with $Z_N(j) = \sum_{\ell = 1}^N x_\ell(j) u_\ell, \ \ 1 \le j \le p$. Assumption \ref{as-2} yields that $a(\cdot)$ is bounded and therefore by Assumptions \ref{as-cor} and \ref{a:xu-indep} we conclude that
\begin{align*}
EZ_N(j)&=\sum_{\ell=1}^N\sum_{k=1}^NE[x_\ell(j)x_k(j)u_\ell u_k]\leq\sum_{\ell=1}^N\sum_{k=1}^N|E[x_\ell(j)x_k(j)]||E[u_\ell u_k]|\\
&\leq O(1)\sum_{\ell=1}^N\sum_{k=1}^N(Ex^2_\ell(j)Ex^2_k(j))^{1/2}(|k-\ell|+1)^{-\kappa}=O(N).
\end{align*}
Since $A^{-1}_N=O_P(1/N)$, we obtain that
\beq\label{beta-1}
N^{1/2}\|\hat\bbe_N - \bbe_0\|=O_P(1).
\eeq
Using Assumption \ref{a:xu-indep} we get via the ergodic theorem that
\beq\label{ergod-1}
\frac{1}{N} \max_{1\leq i \leq N}\left\|\sum_{j=1}^i\bx_j-\frac{i}{N}\sum_{j=1}^N\bx_j   \right\|=o_P(1).
\eeq
It follows from \eqref{beta-1} and \eqref{ergod-1} that
\begin{align*}
N^{-1/2}\max_{1\leq i \leq N}\left|\left(\sum_{j=1}^i\hat{u}_j-\frac{i}{N}\sum_{j=1}^N\hat{u}_j\right)-\left(\sum_{j=1}^i{u}_j-\frac{i}{N}\sum_{j=1}^N{u}_j\right)\right|=o_P(1),
\end{align*}
completing the proof of Theorem \ref{p:lr}.
\qed

\medskip
\noindent
{\it Proof of Theorem \ref{p:nr}.} First we write
$$
L_N(\btheta)=\sum_{i=1}^N u_i^2+V_N(\btheta),\quad V_N(\btheta)=2\sum_{i=1}^Nu_i(h(\bx_i,\btheta_0)-h(\bx_i, \btheta))+\sum_{i=1}^N(h(\bx_i,\btheta_0)-h(\bx_i, \btheta))^2
$$
and  the location of the minimum of $L_N$ and $V_N$ is the same. Using Assumption \ref{a:h} and the ergodic we get that
$$
\sup_{\btheta \in \bTheta}\left|  \frac{1}{N}L_N(\btheta)-V(\btheta) \right|=o(1)\;\;\;\;a.s.,\;\;\;\;\mbox{where}\;\;\;V(\btheta)=E(h(\bx_0,\btheta_0)-h(\bx_0,\btheta))^2.
$$
Since $V(\btheta)$ has a unique maximum at $\btheta_0$, standard arguments yield (c.f.\ Pfanzagl (1994)) that
\beq\label{cos-th}
\hat{\btheta}_N\to\btheta_0\;\;\;\;a.s.
\eeq
Next we observe that
$$
\frac{\partial }{\partial\btheta}L_N(\hat{\btheta}_N)={\bf 0}.
$$
Also, by the ergodic theorem and Assumption \ref{a:h} we have that
$$
\sup_{\btheta \in \bTheta}\left\|\frac{\partial^2}{\btheta^2}\frac{1}{N}L_N(\btheta)-\frac{\partial^2}{\btheta^2}V(\btheta)\right\|=o(1)\;\;a.s.,
$$
$\partial^2 V(\btheta)/\btheta^2$ is continuous on $\bTheta$ and $\partial^2 V(\btheta_0)/\btheta^2$ is nonsingular since $V(\theta)$ has a unique
minimum at $\btheta_0$. Applying the mean value theorem coordinatewise we obtain that
$$
\frac{\partial}{\partial\btheta}L_N(\btheta_0)=
\frac{\partial}{\partial\btheta}L_N(\btheta_0) -\frac{\partial}{\partial\btheta}L_N(\hat{\btheta}_N)=\bG_N(\btheta_0-\hat{\btheta}_N)^\T
$$
and
$$
\frac{1}{N}\bG_N\to \frac{\partial^2}{\btheta^2}V(\btheta_0)\;\;\;a.s.
$$
Following the proof of Theorem \ref{p:lr} one can verify that
$$
E\left\|\frac{\partial}{\partial\btheta}L_N(\btheta_0)\right\|^2=O(N)
$$
and therefore
\beq\label{non-1}
N^{1/2}\left\| \hat{\btheta}_N-\btheta_0\right\|=O_P(1).
\eeq
Using a two term Taylor expansion with the ergodic theorem and \eqref{non-1}
we get that
$$
\sum_{j=1}^i\hat{u}_j-\frac{i}{N}\sum_{j=1}^N\hat{u}_j=\sum_{j=1}^i{u}_j-\frac{i}{N}\sum_{j=1}^N{u}_j+{\mathcal R}_{i,1}+{\mathcal R}_{i,2}
$$
with
$$
{\mathcal R}_{i,1}=\left(\sum_{j=1}^i\frac{\partial}{\btheta}{h}(\bx_j,\btheta_0)-\frac{i}{N}\sum_{j=1}^N\frac{\partial}{\btheta}{h}(\bx_j,\btheta_0)\right)^\T(\btheta_0-\hat{\btheta}_N)
$$
and $N^{-1/2}\max_{1\leq i\leq N}|{\mathcal R}_{i,2}|=o_P(1)$.  Repeating the argument used in the proof of Theorem \ref{p:lr}, by \eqref{non-1} and the ergodic theorem we obtain that $N^{-1/2}\max_{1\leq i\leq N}|{\mathcal R}_{i,1}|=o_P(1)$. Hence the proof is complete.
\qed

\medskip

 \subsection{Proofs of the results in Section \ref{s:comp}}~\\

 \begin{lemma}\label{eig-lem-1} If Assumptions \ref{as-1}--\ref{as-3}, \ref{mu-con}, \ref{ass-un} are satisfied and $\{e_i, -\infty<i<\infty\}$ is a stationary and ergodic sequence with $Ee_0=0$ and $Ee^2_0=\sigma^2$, then we have that
 \begin{equation}\label{H-1}
 \max_{1\leq k \leq N}\left|H_N(k/N)-\frac{1}{N}\left( \sum_{\ell=1}^ku^2_\ell  +\sum_{\ell=1}^k(\mu_\ell-\bar{\mu})^2  \right)\right|=O_P(N^{-1/2}\log N)
 \end{equation}
 and
 \begin{equation}\label{H-2}
 \int_0^1\left( H_N(u)-(b(u)+g(u))   \right)^2dt=o_P(1).
 \end{equation}
 \end{lemma}
 \begin{proof}

  It is easy to see that
 \begin{align*}
 NH_N(k/N)=&\sum_{\ell=1}^ku^2_\ell +\sum_{\ell=1}^k(\mu_\ell-\bar{\mu})^2+2\sum_{\ell=1}^ku_\ell(\mu_\ell-\bar{\mu})+2(\bar{\mu}-\bar{X}_N)\sum_{\ell=1}^ku_\ell\\
&\hspace{1cm} +2(\bar{\mu}-\bar{X}_N)\sum_{\ell=1}^k(\mu_\ell-\bar{\mu})+k(\bar{\mu}-\bar{X}_N)^2.
 \end{align*}
 It follows from Theorem \ref{basic} that
 $$
 \bar{X}_N-\bar{\mu}=\frac{1}{N}\sum_{\ell=1}^Nu_\ell=O_P(N^{-1/2})
 $$
 and
 $$
 \max_{1\leq k \leq N}\left|\sum_{\ell=1}^k u_\ell\right|=O_P(N^{1/2})
 $$
 and therefore
 $$
 \max_{1\leq k \leq N}\left| (\bar{\mu}-\bar{X}_N)\sum_{\ell=1}^ku_\ell  \right|=O_P(1),\quad \max_{1\leq k \leq N}\left| (\bar{\mu}-\bar{X}_N)\sum_{\ell=1}^k(\mu_\ell-\bar{\mu})\right|=O_P(N^{1/2})
 $$
 and
 $$
\max_{1\leq k \leq N} k(\bar{\mu}-\bar{X}_N)^2=O_P(1).
 $$
 Using Assumptions \ref{as-2}, \ref{mu-con} and  \ref{ass-un}, we get that
 $$
 E\left( \sum_{\ell=m}^ku_\ell(\mu_\ell-\bar{\mu}) \right)^2\leq c_1(k-m)\;\;\mbox{for all}\;\;1\leq k \leq m\leq N
 $$
 with some $c_1$ and therefore by Menshov's inequality (cf.\ Billingsley (1968, p.\ 102)) that
 $$
 \max_{1\leq k \leq N}\left|\sum_{\ell=1}^ku_\ell(\mu_\ell-\bar{\mu})  \right|=O_P(N^{1/2}\log N).
 $$
 Next we show that
 \begin{equation}\label{abel}
 \max_{1\leq k \leq N}\left|\sum_{\ell=1}^k(u^2_\ell-\sigma^2a^2(\ell/N))\right|=o_P(N).
 \end{equation}
 Set
 $$
\cS(0)=0\;\;\mbox{ and }\;\;\cS(k)=\sum_{\ell=1}^k(e_\ell^2-\sigma^2).
 $$
 By Abel's summation formula we have
 \begin{align*}
 \sum_{\ell=1}^k(u^2_\ell-\sigma^2a^2(\ell/N))&=\sum_{\ell=1}^k a^2(\ell/N)(\cS(\ell)-\cS(\ell-1))\\
 &=\cS(N)a^2(1)-\sum_{\ell=1}^{k-1}\cS_\ell(a^2((\ell+1)/N)-a^2(\ell/N)).
 \end{align*}

 It follows from the ergodic theorem (cf.\ Breiman (1968, p.\ 118) that
 $$
 \lim_{k\to \infty}\frac{1}{k}\cS(k)=0\quad \mbox{a.s.}
 $$
 For any $\delta>0$, there is a random variable $k^*=k^*(\omega)$ such that $|\cS(k)|\leq \delta k$ if $k\leq k^*$ and therefore for $k>k^*$
 \begin{align*}
 \frac{1}{N}&\left|\sum_{\ell=1}^{k-1}\cS_\ell(a^2((\ell+1)/N)-a^2(\ell/N))\right|\\
 &\leq
 \frac{1}{N}\sum_{\ell=1}^{k^*}|\cS_\ell(a^2((\ell+1)/N)-a^2(\ell/N))|+\frac{1}{N}\sum_{\ell=k^*+1}^{k-1}\frac{|\cS_\ell|}{\ell}\ell|a^2((\ell+1)/N)-a^2(\ell/N)|\\
 &=o_P(1)+\delta\sum_{\ell=1}^{N-1}|a^2((\ell+1)/N)-a^2(\ell/N)|.
 \end{align*}
 It follows from Assumption \ref{as-2} that $a^2(t)$ also has bounded variation on $[0,1]$. Since $\delta$ can be as small as we want, the proof of \eqref{abel} is complete. Observing that
 $$
 \int_0^1\left(\frac{1}{N}\sum_{\ell=1}^{\lf Nt\rf}a^2(\ell/N)-\int_0^t a^2(s)ds\right)^2dt=o(1),
 $$
 the proof of \eqref{H-2} is complete.
 \end{proof}

 \medskip
 \noindent
 {\it Proof of Theorem \ref{under-th}.}  It is an immediate consequence of Lemma \ref{eig-lem-1} and the definition of $\widehat{C}_N(t,s)$.
\qed

It follows from the definition of $\hat{\gamma}_{N;k,\ell}$ that for $\ell\geq 0$
\begin{align*}
N\hat{\gamma}_{N;k,\ell}=r_{k,\ell,1}+\ldots +r_{k,\ell,9},
\end{align*}
where
\begin{align*}
&r_{k,\ell,1}=\sum_{i=1}^{k-\ell}(X_i-\mu_i)(X_{i+\ell}-\mu_{i+\ell}),&&r_{k,\ell,2}=(\bar{\mu}-\bar{X}_N)\sum_{i=1}^{k-\ell}(X_i-\mu_i)\\
&r_{k,\ell,3}=(\bar{\mu}-\bar{X}_N)\sum_{i=1}^{k-\ell}(X_{i+\ell}-\mu_{i+\ell}),&&r_{k,\ell,4}=(k-\ell)(\bar{\mu}-\bar{X}_N)^2\\
&r_{k,\ell,5}=\sum_{i=1}^{k-\ell}(X_i-\mu_i)(\mu_{i+\ell}-\bar{\mu}),&&r_{k,\ell,6}=\sum_{i=1}^{k-\ell}(X_{i+\ell}-\mu_{i+\ell})(\mu_{i}-\bar{\mu})\\
&r_{k,\ell,7}=(\bar{\mu}-\bar{X}_N)\sum_{i=1}^{k-\ell}(\mu_i-\bar{\mu}),&&r_{k,\ell,8}=(\bar{\mu}-\bar{X}_N)\sum_{i=1}^{k-\ell}(\mu_{i+\ell}-\bar{\mu})
\end{align*}
and
$$
r_{k,\ell,9}=\sum_{i=1}^{k-\ell}(\mu_i-\bar{\mu})(\mu_{i+\ell}-\bar{\mu}).
$$
\begin{lemma}\label{lolemnull} If $H_0$,  Assumptions \ref{as-1}, \ref{as-2} and \ref{ass-k}--\ref{as-ber} are satisfied, then we have that
$$
\int_0^1(\hat{g}_N(t)-b(t))^2=o_P(1).
$$
\end{lemma}

\begin{proof}

It is easy to see that
 \begin{align*}
 Er_{k,\ell,1}=\sum_{i=1}^{k-\ell}Eu_iu_{i+\ell}=\sum_{i=1}^{k-\ell}a(i/N)a((i+\ell)/N)r(\ell),\;\;\mbox{where} \;r(\ell)=Ee_0e_\ell,
 \end{align*}
 resulting in
 \begin{align*}
 E\sum_{\ell=1}^kK(\ell/h)\frac{1}{N}r_{k,\ell,1}=\frac{1}{N}\sum_{\ell=1}^kK(\ell/h)r(\ell)\sum_{i=1}^{k-\ell}a(i/N)a((i+\ell)/N).
 \end{align*}
 Let $M$ be a positive integer. We have  for $k\geq M$ that
 \begin{align*}
F_N(k)&=\biggl|\sum_{\ell=1}^kK(\ell/h)r(\ell) \frac{1}{N} \sum_{i=1}^{k-\ell}a(i/N)a((i+\ell)/N)-\sum_{\ell=1}^MK(\ell/h)r(\ell) \frac{1}{N} \sum_{i=1}^{k-\ell}a(i/N)a((i+\ell)/N)\biggl|\\
&\leq c_1\sum_{\ell=M+1}^\infty |r(\ell)|
 \end{align*}
 with some constant $c_1$ since by Assumptions \ref{as-2} and \ref{ass-k}  $K$ and $a$ are bounded functions.
  If $1\leq k\leq M$, then
 \begin{align*}
 \Biggl|\sum_{\ell=1}^kK(\ell/h)r(\ell) \frac{1}{N} \sum_{i=1}^{k-\ell}a(i/N)a((i+\ell)/N)\Biggl|\leq c_2\frac{M^2}{N}.
 \end{align*}
 Hence  we have
\begin{align}\label{ker-1}
\max_{1\leq k \leq N}|F_N(k)|\leq c_1\sum_{\ell=M+1}^\infty |r(\ell)|+  c_2\frac{M^2}{N}.
\end{align}
Using Assumptions \ref{as-2} and \ref{ass-k} we have
\begin{align}\label{ker-2}
\max_{1\leq k \leq N}\Biggl| \sum_{\ell=1}^{M}(K(\ell/h)-1)r(\ell) \frac{1}{N} \sum_{i=1}^{k-\ell}a(i/N)a((i+\ell)/N)\Biggl|\leq c_3\frac{M^2}{h}
\end{align}
with some constant $c_3$
and
\begin{align}\label{ker-3}
\max_{1\leq k \leq N}&\Biggl| \sum_{\ell=1}^{M} r(\ell) \frac{1}{N} \left(  \sum_{i=1}^{k-\ell}a(i/N)a((i+\ell)/N)- \sum_{i=1}^{k}a(i/N)a((i+\ell)/N)  \right)\Biggl|\\
&\leq c_4\left(\sum_{\ell=1}^\infty |r(\ell)|\right)\frac{M}{N}.\notag
\end{align}
Next we note that
\begin{align}\label{ker-4}
\max_{1\leq k\leq N}&\left|   \sum_{\ell=1}^{M} r(\ell) \frac{1}{N}\sum_{i=1}^{k}a(i/N)(a((i+\ell)/N)-a(i/N)) \right|\\
&\leq c_5\frac{1}{N}\max_{1\leq \ell\leq N}\sum_{i=1}^N |a((i+\ell)/N)-a(i/N)|\notag\\
&\leq c_6 \frac{M}{N},\notag
\end{align}
by Assumption \ref{as-2}, where $c_5$ and $c_6$ are constants and $a(u)=0$ for $u>1$. Finally,
\begin{align}\label{ker-5}
\max_{1\leq k\leq N}\left|\sum_{\ell=1}^{M} r(\ell) \frac{1}{N}\sum_{i=1}^{k}a^2(i/N)-\sum_{\ell=1}^{\infty} r(\ell) \frac{1}{N}\sum_{i=1}^{k}a^2(i/N)\right|\leq
c_7\sum_{\ell=M+1}^\infty |r(\ell)|,
\end{align}
with some constant $c_7$. Putting together \eqref{ker-1}--\eqref{ker-5} we conclude that
\begin{align*}
\limsup_{N\to \infty}\left|
 E\sum_{\ell=1}^kK(\ell/h)\frac{1}{N}r_{k,\ell,1}-\left(\sum_{\ell=1}^\infty r(\ell)\right)\frac{1}{N}\sum_{i=1}^k a^2(i/N)\right|\leq (c_1+c_7)\sum_{M+1}^\infty|r(\ell)|,
\end{align*}
and since we can take $M$ as large as we want to we obtain
\begin{align*}
\lim_{N\to \infty}\left|
 E\sum_{\ell=1}^kK(\ell/h)\frac{1}{N}r_{k,\ell,1}-\left(\sum_{\ell=1}^\infty r(\ell)\right)\frac{1}{N}\sum_{i=1}^k a^2(i/N)\right|=0.
\end{align*}
Next we show that
\begin{equation}\label{ker-6}
\frac{1}{N}\sum_{k=1}^N\left(\sum_{\ell=1}^k K(h/\ell)\frac{1}{N}(r_{k,\ell,1}-E r_{k,\ell,1})    \right)^2=o_P(1).
\end{equation}
We observe that by the stationarity of the $e_i$'s and Assumptions \ref{as-2} and \ref{ass-k}
\begin{align*}
E\frac{1}{N}&\sum_{k=1}^N\left(\sum_{\ell=1}^k K(h/\ell)\frac{1}{N}(r_{k,\ell,1}-E r_{k,\ell,1})    \right)^2\\
&=\frac{1}{N^3}\sum_{k=1}^N\sum_{\ell=1}^k\sum_{\ell'=1}^kK(\ell/h)K(\ell'/h)\sum_{i=1}^{k-\ell}\sum_{j=1}^{k-\ell'}a(\ell/N)a((\ell+i)/N)a(j/N)a((j+\ell')/N)\\
&\hspace{1cm}\times(E
e_ie_{i+\ell}e_{j}e_{j+\ell'}-r(\ell)r(\ell'))\\
&\leq c_8\frac{1}{N^3}\sum_{k=1}^N\sum_{\ell=1}^{ch}\sum_{\ell'=1}^{ch}k\sum_{j=1}^{k}|E
e_0e_{\ell}e_{j}e_{j+\ell'}-r(\ell)r(\ell')|\\
&\leq c_8\frac{1}{N^2}\sum_{k=1}^N\sum_{\ell=1}^{ch}\sum_{\ell'=1}^{ch}\sum_{j=1}^{k}|E
e_0e_{\ell}e_{j}e_{j+\ell'}-r(\ell)r(\ell')|.
\end{align*}
Let
\begin{align*}
G_{1,k}&=\{(j,\ell, \ell'):\; ch+1\leq j\leq k, 1\leq\ell, \ell'\leq ch\}\\
G_{2,k}&=\{(j,\ell, \ell'):\; 1\leq j\leq ch, 1\leq\ell, \ell'\leq ch\}
\end{align*}
and define
$$
Q_{1,k}=\sum_{(j,\ell,\ell')\in G_{1,k}}|E
e_0e_{\ell}e_{j}e_{j+\ell'}-r(\ell)r(\ell')|,  \quad Q_{2,k}=\sum_{(j,\ell,\ell')\in G_{1,k}}|E
e_0e_{\ell}e_{j}e_{j+\ell'}-r(\ell)r(\ell')|.
$$
Next we define
$$
\bae_{j,j-\ell-1}=f(\vare_j, \vare_{j-1}, \ldots ,\vare_{\ell+1}, \vare_{\ell}', \vare_{\ell-1}',\ldots )
$$
and
$$
\bae_{j+\ell',j+\ell'-\ell-1}=f(\vare_{j+\ell'}, \vare_{j+\ell'-1}, \ldots ,\vare_{\ell+1}, \vare_{\ell}',\vare_{\ell-1}',\ldots ),
$$
where $\vare_v', -\infty<v<\infty$ are independent copies of $\vare_0$, independent of $\vare_j, -\infty<j<\infty$.  It follows from Assumption \ref{as-ber} that
$(e_0, e_\ell)$ is independent of $(\bae_{j,j-\ell-1}, \bae_{j+\ell', j+\ell'-\ell-1})$. Also, according to the construction the vectors $(e_j,e_{j+\ell'})$ and
$(\bae_{j,j-\ell-1}, \bae_{j+\ell', j+\ell'-\ell-1})$ have the same distribution. Note that $Ee_0e_\ell \bae_{j,j-\ell-1}\bae_{j+\ell', j+\ell'-\ell-1}=
E[e_0e_\ell ]E[\bae_{j,j-\ell-1}\bae_{j+\ell', j+\ell'-\ell-1}]=r(\ell)r(\ell')$.  Hence
$$
Ee_0e_\ell e_{j}e_{j+\ell'}-r(\ell)r(\ell')=Ee_0e_\ell[e_je_{j+\ell'}-\bae_{j,j-\ell-1}\bae_{j+\ell',j+\ell'-\ell-1}].
$$
It follows from Assumption \ref{as-ber}
\begin{align*}
(E(e_j-\bae_{j,j-\ell-1})^4)^{1/4}\leq c_8(j-\ell)^{-\alpha}\;\;\mbox{and}\;\;(E(e_{j+\ell'}-\bae_{j+\ell',j+\ell'-\ell-1})^4)^{1/4}\leq c_8(j+\ell'-\ell)^{-\alpha}
\end{align*}
with some constant $c_8$ for all $(j,\ell,\ell')\in G_{1,k}$. Hence the cauchy--Schwartz inequality yields
\begin{align*}
|E&e_0e_\ell[e_je_{j+\ell'}-\bae_{j,j-\ell-1}\bae_{j+\ell',j+\ell'-\ell-1}]|\\
&E|e_0e_\ell e_{j}[e_{j+\ell'}-\bae_{j+\ell', j+\ell'-\ell-1}]|+E|e_0e_\ell \bae_{j+\ell',j+\ell'-\ell-1}[e_{j}-\bae_{j,j-\ell-1}]|\\
&\leq (Ee_0^4e_\ell^4 e_{j}^4E[e_{j+\ell'}-\bae_{j+\ell', j+\ell'-\ell-1}]^4)^{1/4}+(Ee_0^4e_\ell^4 \bae_{j+\ell',j+\ell'-\ell-1}^4E[e_{j}-\bae_{j,j-\ell-1}]^4)^{1/4}\\
&\leq c_8(Ee_0^4)^{3/4}((j-\ell)^{-\alpha}+(j+\ell'-\ell)^{-\alpha})\\
&\leq 2c_8(Ee_0^4)^{3/4}(j-\ell)^{-\alpha}
\end{align*}
for all $(j,\ell,\ell')\in G_{1,k}$. Thus we get with $c_9=2c_8(Ee_0^4)^{3/4}$ that
\begin{align*}
Q_{1,k}\leq c_9 \sum_{(j,\ell,\ell')\in G_{1,k}}(j-\ell)^{-\alpha}\leq c_{10}h\int_{ch+1}^\infty\int_{1}^{ch}(x-y)^{-\alpha}dydx\leq c_{11}h
\end{align*}
with some constants $c_{10}$ and $c_{11}$. \\
We note that
\begin{align*}
Q_{2,k}\leq \sum_{(j,\ell,\ell')\in G_{2,k}}|E[e_0e_{\ell}e_{j}e_{j+\ell'}]|+\sum_{(j,\ell,\ell')\in G_{2,k}}|E[e_0e_{\ell}]E[e_{j}e_{j+\ell'}]|
\end{align*}
and
$$
\sum_{(j,\ell,\ell')\in G_{2,k}}|E[e_0e_{\ell}]E[e_{j}e_{j+\ell'}]|\leq ch\left(\sum_{\ell=1}^\infty |Ee_0e_\ell|\right)^2.
$$
Let $e_{j,m}$ be the random variables defined in Assumption \ref{as-ber}. We get for all $0\leq s \leq t \leq v\leq 2ch $ that
\begin{align*}
e_0&e_se_te_v=e_0e_{s,s}(e_t-e_{t,t-s})(e_v-e_{v,v-t})+e_0e_{s,s}(e_t-e_{t,t-s})e_{v,v-t}+e_0e_{s,s}e_{t,t-s}(e_v-e_{v,v})\\
&+e_0e_{s,s}e_{t,t-s}e_{v,v}+e_0(e_s-e_{s,s})(e_t-e_{t,t-s})e_v+e_0(e_s-e_{s,s})e_{t,t-s}e_{v,v-t}\\
&+e_0(e_s-e_{s,s})e_{t,t-s}(e_v-e_{v,v-t}).
\end{align*}
The definition of $e_{j,m}$ yields that $e_{v,v-t}$ is independent of $e_0e_{s,s}(e_t-e_{t,t-s})$, $e_0$ is independent of $e_{s,s}e_{t,t-s}e_{v,v}$ and $e_{v,v-t}$ is independent of $e_0(e_s-e_{s,s})e_{t,t-s}$ and therefore
$$
E[e_0e_{s,s}(e_t-e_{t,t-s})e_{v,v-t}]=0,\;\;E[e_0e_{s,s}e_{t,t-s}e_{v,v}]=0\;\;\mbox{and}\;\;E[e_0(e_s-e_{s,s})e_{t,t-s}e_{v,v-t}]=0.
$$
Using Assumption \ref{as-ber} we obtain that
\begin{align*}
\sum_{1\leq s\leq t\leq v \leq 2ch}&E|e_0e_{s,s}(e_t-e_{t,t-s})(e_v-e_{v,v-t})|\\
&\leq (Ee_0^4)^{1/2}\sum_{1\leq s\leq t\leq v \leq 2ch}(E(e_t-e_{t,t-s})^4)^{1/4}(E(e_v-e_{v,v-t})^4)^{1/4}\\
&=O(h).
\end{align*}
Similarly,
$$
\sum_{1\leq s\leq t\leq v \leq 2ch}E|e_0(e_s-e_{s,s})(e_t-e_{t,t-s})e_v|=O(h)
$$
and
$$
\sum_{1\leq s\leq t\leq v \leq 2ch}E|e_0(e_s-e_{s,s})e_{t,t-s}(e_v-e_{v,v-t})|=O(h)
$$
and
$$
\sum_{1\leq s\leq t\leq v \leq 2ch}E|e_0e_{s,s}e_{t,t-s}(e_v-e_{v,v})|=O(1)\int_0^{2ch}\int_s^{2ch}\int_t^{2ch}v^{-\alpha}dvdtds=O(h).
$$
Thus we conclude
$$
Q_{2,k}\leq c_{l2}h,
$$
which completes the proof of
$$
E\frac{1}{N}\sum_{k=1}^N\left(\sum_{\ell=1}^k K(h/\ell)\frac{1}{N}(r_{k,\ell,1}-E r_{k,\ell,1})    \right)^2=O(h/N)
$$
and therefore \eqref{ker-6} follows via Markov's inequality.\\
 Theorem \ref{basic} yields that
\begin{align*}
\max_{1\leq k \leq N}\left|\sum_{\ell=1}^{k-1}K(\ell/h)\frac{1}{N}r_{k,\ell,2} \right|\leq \frac{ch}{N^2}\left|\sum_{i=1}^Nu_i\right|\max_{ |u |\leq c}|K(u)|
\max_{1\leq k\leq N}\left|\sum_{i=1}^k u_i\right|=O_P(h/N).
\end{align*}
Similar argument gives
\begin{align*}
\max_{1\leq k \leq N}\left|\sum_{\ell=1}^{k-1}K(\ell/h)\frac{1}{N}r_{k,\ell,3} \right|=O_P(h/N).
\end{align*}
Theorem \ref{basic} and Assumption \ref{ass-k} yield
\begin{align*}
\max_{1\leq k \leq N}\left| \sum_{\ell=1}^kK(\ell/N)\frac{1}{N}r_{k,\ell,4}\right|\leq  \frac{1}{N^3}\left( \sum_{i=1}^N u_i\right)^2 \max_{1\leq k\leq N}
\left|\sum_{\ell=1}^k K(\ell/h)(k-\ell)\right|=O_P(h/N).
\end{align*}
Since $r_{k,\ell,5}=r_{k,\ell,6}=r_{k,\ell,7}=r_{k,\ell,8}=r_{k,\ell,9}=0$ under $H_0$ we proved that
\begin{align}\label{fi-1}
\int_0^1\left(  \sum_{\ell=1}^{\lf Nu\rf-1}K(\ell/h)\hat{\gamma}_{N,\lf Nu\rf,\ell} -\sum_{\ell=1}^\infty r(\ell)\int_0^ua^2(v)dv\right)^2du=o_P(1).
\end{align}
Similar arguments show that
\begin{align}\label{fi-2}
\int_0^1\Biggl(  \sum^{0}_{\ell=-(\lf Nu\rf-1)}K(\ell/h)\hat{\gamma}_{N,\lf Nu\rf,\ell} -\sum^{0}_{\ell=-\infty} r(\ell)\int_0^ua^2(v)dv\Biggl)^2du=o_P(1).
\end{align}
Now Lemma \ref{lolemnull} follows from \eqref{fi-1} and \eqref{fi-2}.
\end{proof}

\medskip
\begin{lemma}\label{lolemalt} If $H_A$,  Assumptions \ref{as-1}, \ref{as-2} and \ref{ass-k}--\ref{as-ber} are satisfied, then we have
$$
\int_0^1\hat{g}_N^2(u)du=O_P(h^2).
$$
If in addition, Assumption
$$
\int_0^1\left(\frac{1}{h}\hat{g}_N(u)-g(u)\int_{-c}^cK(v)dv\right)^2du=o_P(1).
$$
\end{lemma}

\begin{proof} Following the proof of Theorem \ref{basic} one can show that
\begin{align*}
\max_{1\leq k\leq N}\max_{1\leq \ell <k}\left| \sum_{i=1}^{k-\ell}(X_i-\mu_i)(\mu_{i+\ell}-\bar{\mu}) \right|
=\max_{1\leq k\leq N}\max_{1\leq \ell <k}\left| \sum_{i=1}^{k-\ell}u_i(\mu_{i+\ell}-\bar{\mu}) \right|=O_P(N^{1/2}),
\end{align*}
and therefore
\begin{align*}
\max_{1\leq k\leq N}\left|  \sum_{\ell=1}^kK(\ell/h)\frac{1}{N}r_{k,\ell,5} \right|=O_P(h/N^{1/2}).
\end{align*}
Similarly,
\begin{align*}
\max_{1\leq k\leq N}\left|  \sum_{\ell=1}^kK(\ell/h)\frac{1}{N}r_{k,\ell,6} \right|=O_P(h/N^{1/2}).
\end{align*}
Combining Theorem \ref{basic} and Assumption \ref{mu-con} we conclude
\begin{align*}
\max_{1\leq k\leq N}\left|  \sum_{\ell=1}^kK(\ell/h)\frac{1}{N}r_{k,\ell,7} \right|&\leq \frac{1}{N^3}\left(\sum_{i=1}^Nu_i\right)^2\max_{1\leq k \leq N} \left| \sum_{\ell=1}^k K(\ell/h)\sum_{j=1}^{k-\ell}(\mu_i-\bar{\mu})\right|\\
&=O_P(h/N)
\end{align*}
and
\begin{align*}
\max_{1\leq k\leq N}\left|  \sum_{\ell=1}^kK(\ell/h)\frac{1}{N}r_{k,\ell,8} \right|=O_P(h/N)
\end{align*}
We note that by Assumptions \ref{mu-con} and \ref{ass-k}
\begin{align*}
\max_{1\leq k\leq N}\left|\sum_{\ell=1}^{k-1}K(\ell/h)\frac{1}{N}\sum_{i=1}^{k-\ell}(\mu_i-\bar{\mu})(\mu_{i+\ell}-\bar{\mu})\right|=O_(h),
\end{align*}
and therefore by the proof of Lemma \ref{lolemnull} we get
\begin{align}\label{fi-3}
\int_0^1\left(  \sum_{\ell=1}^{\lf Nu\rf-1}K(\ell/h)\hat{\gamma}_{N,\lf Nu\rf,\ell} \right)^2du=O_P(h^2).
\end{align}
Similarly,
\begin{align}\label{fi-4}
\int_0^1\Biggl(  \sum^{0}_{\ell=-(\lf Nu\rf-1)}K(\ell/h)\hat{\gamma}_{N,\lf Nu\rf,\ell} \Biggl)^2=O_P(h^2),
\end{align}
and now the first part of Lemma \ref{lolemalt} follows from \eqref{fi-3} and \eqref{fi-4}.\\
Using Assumption \ref{as-sm} we conclude
\begin{align*}
\max_{1\leq k\leq N}\left|\sum_{\ell=1}^{k-1}K(\ell/h)\frac{1}{N}\sum_{i=1}^{k-\ell}(\mu_i-\bar{\mu})(\mu_{i+\ell}-\mu_i)\right|=o(h),
\end{align*}
and by Assumption \ref{mu-con}
$$
\max_{1\leq k\leq N}\left|\sum_{\ell=1}^{k-1}K(\ell/h)\frac{1}{N}\sum^k_{i=k-\ell+1}(\mu_i-\bar{\mu})^2\right|=o(1).
$$
It follows from Assumption \ref{ass-k} that for all $0<u<1$,
$$
\frac{1}{h}\sum_{\ell=1}^{\lf Nu\rf -1}K(\ell/h)\to\int_0^cK(u)du.
$$
Thus we conclude
\begin{align*}
\int_0^1\left(\frac{1}{h}\sum_{\ell=1}^{\lf Nu\rf -1}K(\ell/h)\frac{1}{N}\sum_{i=1}^{\lf Nu\rf-\ell}(\mu_i-\bar{\mu})(\mu_{i+\ell}-\bar{\mu})-g(u)\int_0^cK(v)dv\right)^2du=o(1),
\end{align*}
and therefore we can replace \ref{fi-3} and \ref{fi-4} with the more precise
\begin{align*}
\int_0^1\left( \frac{1}{h} \sum_{\ell=1}^{\lf Nu\rf-1}K(\ell/h)\hat{\gamma}_{N,\lf Nu\rf,\ell} -g(u)\int_0^c K(v)dv\right)^2du=o_P(1)
\end{align*}
and
\begin{align*}
\int_0^1\Biggl( \frac{1}{h} \sum^{0}_{\ell=-(\lf Nu\rf-1)}K(\ell/h)\hat{\gamma}_{N,\lf Nu\rf,\ell} -
g(u)\int^0_{-c}K(v)dv\Biggl)^2du=o_P(1).
\end{align*}
This completes the proof of Lemma \ref{lolemalt}.
\end{proof}
 \medskip

\noindent
{\it Proof of Theorem \ref{second-th}.} The result in \eqref{lo-1} follows from Lemma \ref{lolemnull} while \ref{lo-1/2} and \ref{lo-2} are immediate consequences of Lemma \ref{lolemalt}.  \qed
 
\end{document}